\documentclass{amsart}

\usepackage{amsmath, latexsym,  amssymb, amsthm, amscd, comment, enumerate}
\usepackage[all]{xy}
\usepackage[british]{babel}
\usepackage{url}

\usepackage{enumitem}

\renewcommand{\sharp}{\#}
\renewcommand{\emptyset}{\varnothing}

\DeclareMathOperator{\codim}{codim}

\DeclareMathOperator{\Gal}{Gal}

\newcommand{\cnbundle}{{\mathcal{O}_{N}}}
\newcommand{\cani}{{\mathcal{O}}}

\newcommand{\torsione}{\zeta}

\newcommand{\acca}{H}
\newcommand{\sotto}{B}

\newcommand{\cVv}{(h(V)+\deg V)}
\newcommand{\cV}{(h(V)+\deg V)[k_\mathrm{tor}(V):k_\mathrm{tor}]}
\newcommand{\cC}{(h(C)+\deg C)[k_\mathrm{tor}(C):k_\mathrm{tor}]}

\newcommand{\elle}{{\mathcal{L}}}

\newcommand{\cuno}{c_1}
\newcommand{\cdue}{c_2}
\newcommand{\cremond}{c_3}
\newcommand{\ctre}{c_4}
\newcommand{\cquattro}{c_5}
\newcommand{\ccinque}{c_6}
\newcommand{\csei}{c_7}

\newcommand{\cod}{\mathrm{cod} \,}

\newcommand{\qe}{\mathbb{Q}}

\renewcommand{\epsilon}{\varepsilon}

\newtheorem{thm}{Theorem}[section]

\newtheorem{con}[thm]{Conjecture}
\newtheorem{propo}[thm]{Proposition}
\newtheorem{lem}[thm]{Lemma}
\newtheorem{cor}[thm]{Corollary}
\newtheorem{D}[thm]{Definition}

\newtheorem*{ZP}{Conjecture 1.1'}

\newcommand{\sbgrpG}{F}
\newcommand{\sbgrpGG}{F}
\newcommand{\sbgrpH}{H_1 }

\title[On torsion anomalous intersections]{On torsion anomalous intersections}

\author{Sara Checcoli}

\author{Francesco Veneziano}

\author{Evelina Viada}

\thanks{All authors are supported by the SNSF}



\begin{document}

\begin{abstract}
A deep conjecture on torsion anomalous varieties states that if $V$ is a weak-transverse variety in an abelian variety, then the complement $V^{ta}$ of all $V$-torsion anomalous varieties is open and dense in $V$. We prove some cases of this conjecture. We show that the $V$-torsion anomalous varieties of relative codimension one are non-dense in any weak-transverse variety $V$ embedded in a product of elliptic curves with CM. We give explicit uniform bounds in the dependence on $V$. As an immediate consequence we prove the conjecture for $V$ of codimension two in a product of CM elliptic curves. We also point out some implications on the effective Mordell-Lang Conjecture. 

\vspace{1cm}

Una importante congettura sulle variet\`a torsione-anomale afferma che se $V$ \`e una variet\`a debolmente-trasversa in una variet\`a abeliana, allora il complementare $V^{ta}$ di tutte le variet\`a $V$-torsione-anomale  \`e aperto e denso in $V$. In questo articolo dimostriamo alcuni casi della  congettura. In par\-ti\-co\-la\-re, mostriamo che  le variet\`a $V$-torsione-anomale di codimensione relativa uno non sono dense in ogni variet\`a $V$ debolmente trasversa, immersa in un prodotto di curve ellittiche con CM. Inoltre diamo stime esplicite e uniformi nella dipendenza da $V$. Come immediata conseguenza otteniamo la suddetta congettura per $V$  di codimensione due in un prodotto di curve ellittiche CM. Infine, evidenziamo alcune implicazioni sulla Congettura di Mordell-Lang Effettiva.
\end{abstract}

\maketitle

\section{Introduction}
In this article, by \emph{variety} we mean an algebraic variety defined  over the algebraic numbers.
We denote by $G$ a semi-abelian variety defined over a number field $k$ and by $k_{\mathrm{tor}}$ the field of definition of the torsion points of $G$. Let $V$ be a  proper subvariety of $G$.  
The variety $V$  is a \emph{translate},  respectively  a \emph{torsion variety},  if it is a finite  union of translates of  proper algebraic subgroups by  points, respectively  by torsion points.

An irreducible variety $V$  is \emph{transverse}, respectively \emph{weak-transverse}, if  it is not contained in any translate, respectively in any torsion variety.

Of course a torsion variety is in particular a translate, and a transverse variety is weak-transverse. In addition transverse implies non-translate, and weak-transverse implies non-torsion.

\medskip

It is a natural problem to investigate  when a  geometric assumption on $V$ is equivalent to the non-density of some special subsets of $V$. Several classical  conjectures, nowadays theorems, are of this nature. For instance, the Manin-Mumford Conjecture, stating that the torsion points are non-dense in non-torsion subvarieties; the Mordell-Lang Conjecture, saying that the points in a subgroup of finite rank are non-dense in a transverse subvariety;  and the Bogomolov Conjecture, asserting that points of sufficiently small height are non-dense in non-torsion subvarieties.  

\medskip

More recently, new questions of similar type have been raised.  The ``Conjecture on the Intersection with Torsion Varieties" (in short CIT) formulated in the semi-abelian case by Zilber and in Shimura varieties by Pink, predicts that, for a weak-transverse  variety $V$, the intersection of $V$ with the union of all algebraic subgroups of codimension at least $\dim V+1$ is non-dense in $V$, thereby generalising the former Manin-Mumford Conjecture.  In this context one defines
$$S_{r}(V)=V \cap \bigcup_{\codim\, \acca \ge r}\acca,$$ where $r$ is a natural number, and $\acca$ runs over all algebraic subgroups of codimension at least $r$; with this definition the conjecture states: 
\begin{con}[CIT, Conjecture on the Intersection with Torsion Varieties]\label{CIT}
Let $V$ be a weak-transverse variety in a (semi-)abelian variety. Then $S_{\dim V+1}(V)$ is non-dense  in $V$.
\end{con}

One of the remarkable features of the CIT is that it also implies another deep result of modern diophantine geometry: the former Mordell-Lang Conjecture. Indeed, it turns out to be equivalent to the assertion that, for a transverse variety $V$, the intersection of $V$ with the union of all algebraic subgroups of codimension at least $\dim V+1$, translated by points in a subgroup $\Gamma$ of finite rank, is non-dense in $V$.
Mirroring the above definition, one sets
$$S_{r}(V,\Gamma)=V \cap \bigcup_{\codim\, \acca \ge r}\acca+\Gamma,$$ where $r$ is a natural number, $\acca$ runs over all algebraic subgroups of codimension at least $r$ and $H+\Gamma=\{x+\gamma \mid x\in \acca, \gamma \in \Gamma\}$. Then, we have the following reformulation also known as Zilber-Pink Conjecture.
\begin{ZP}[CIT, second version] If $V$ is transverse and $\Gamma$  has finite rank then $S_{\dim V+1}(V, \Gamma)$ is non-dense in $V$.\end{ZP}

This conjecture has been tackled from several points of view, but it has only been  partially answered. For instance it is known for curves in some semi-abelian varieties. A first approach, introduced in 1999 by E.~Bombieri, D.~Masser and U.~Zannier \cite{BMZ_tori} in the multiplicative case for a transverse curve $C$ and trivial $\Gamma$, consists in a two-step argument where one proves that:
\begin{enumerate}
\item The height  of $S_2(C)$ is bounded from above;
\item Bounded height implies bounded degree, hence finiteness by Northcott's theorem.
 \end{enumerate}
 This kind of approach extends to abelian varieties with CM, as recalled later.

\medskip
In higher dimension, several deep problems arise. 
Bombieri, Masser and Zannier \cite{BMZ1} give a new approach to treat this general case. They introduce the notions of anomalous and torsion anomalous subvarieties. In their definitions they always avoid points. For us points can be torsion anomalous, but not anomalous; in particular, $V$-torsion anomalous points are contained in $S_{\dim V+1}(V)$. Our different definition is justified  by the fact that we obtain a  perfect match with the CIT conjecture, as clarified below. 

We have the following definitions.

\begin{D}
An irreducible subvariety $Y$ of $V$ is a \emph{$V$-torsion anomalous} variety if 
\begin{itemize}
\item[-] $Y$ is an irreducible component of $V\cap (B+\zeta)$ with $B+\zeta$ an irreducible  torsion variety;
\item[-]  the dimension of $Y$ is larger than expected, i.e.  $$\codim Y < \codim V +  \codim  B.$$
\end{itemize}

The variety $B+\zeta$ is \emph{minimal} for $Y$ if it satisfies the above conditions and has minimal dimension.
The \emph{relative codimension} of $Y$ is the codimension of $Y$ in its minimal $B+\zeta$.

 We say that $Y$ is a \emph{maximal $V$-torsion anomalous} variety if it is $V$-torsion anomalous and it is not contained in any $V$-torsion anomalous variety of strictly larger dimension.

The complement in $V$ of the union of all $V$-torsion anomalous varieties is denoted by $V^{ta}$.   Clearly $V^{ta}$ is obtained removing from $V$ all maximal $V$-torsion anomalous varieties.  

\end{D}

 Again,  Bombieri, Masser and Zannier use the same notation to indicate the complement of  the union of all $V$-torsion anomalous varieties of positive dimension. 

\medskip

Furthermore, an irreducible variety $Y$ of positive dimension is  \emph{$V$-anomalous} if it is a component of $V\cap (B+p)$ with  $B+p$ an irreducible  translate and in addition $Y$ has dimension larger than expected. The complement in $V$ of the union of all $V$-anomalous varieties is denoted by $V^{oa}.$

 Clearly  points should be excluded from the definition of  $V$-anomalous varieties because otherwise all points would be anomalous, making the notion uninteresting. On the other hand,  as previously remarked, we allow points to be torsion anomalous varieties: they are exactly the torsion anomalous varieties which are not anomalous.

  It may be possible that only some components of $V\cap (B+\zeta)$ are anomalous, so each component has to be treated separately. This justifies the assumption of $Y$ being irreducible. 
  
\medskip

 In the Torsion Openness Conjecture \cite{BMZ1}, Bombieri, Masser and Zannier    conjecture that the complement  of the set of the torsion anomalous varieties of positive dimension is open.  In addition, in the Torsion Finiteness Conjecture, they claim that there are only finitely many maximal  torsion anomalous points. Here we state a slightly stronger conjecture, which includes both their conjectures. In addition, it specifies that $V^{ta}$ is empty exactly when $V$ is not weak-transverse. 
 \begin{con}[Bombieri-Masser-Zannier]
\label{TORAN}
Let $V$ be a weak-transverse variety in a (semi-)abelian variety. Then $V^{ta}$ is a dense open subset of  $V$.
\end{con}

In other words we say that there are only finitely many maximal torsion anomalous varieties of any dimension, as proved below. This immediately shows that Conjecture~\ref{TORAN} implies the Torsion Openness and the Torsion
Finiteness Conjectures. We give a proof of this implication in Proposition~\ref{EquivalenzaCongetture}.

For a hypersurface, Conjecture~\ref{TORAN} is clearly true. Indeed, the intersection of an irreducible torsion variety $\sotto+\zeta$ with a hypersurface is either the variety $\sotto+\zeta$ itself or it has the right dimension $\dim \sotto-1$. So the only $V$-torsion anomalous varieties are torsion varieties contained in $V$; but we know by the Manin-Mumford Conjecture that the  maximal torsion varieties contained in $V$ are finitely many. 

\medskip

Among other results, Bombieri, Masser and Zannier in \cite{BMZ1}, Theorem 1.7,  prove  the openness of $V^{ta}$ for an irreducible variety $V$ of codimension $2$ in $\mathbb{G}^n_m$. 
{One of the}  main ingredients in their proof
is a result of Ax consisting in the analogue of Schanuel's Conjecture in fields
of complex power series in several variables.

\bigskip

The conjecture is easily proven for weak-transverse translates in abelian varieties. In this case there are no torsion anomalous varieties, see Proposition \ref{wtrans}. 
 
 \bigskip
 
In this article we split the set of all maximal $V$-torsion anomalous subvarieties $Y$ of $V$ according to their relative codimension. We study in particular those of relative codimension one: these are indeed the most anomalous non-torsion subvarieties of $V$, hence the most likely to be non-dense in $V$. From this point of view, the article is a natural starting point in the study of torsion anomalous subvarieties. We emphasize that, for any relative codimension, an effective generalisation of our main theorem seems to be completely out of reach at the moment. Indeed, this would imply the effective Mordell-Lang Conjecture, which is well known to be a very hard problem. For higher relative codimension, the method generalises only to weaker results, due to the use of a weaker, but essentially optimal, lower bound for the height in the context of the Bogomolov and Lehmer problems. Thus, a generalisation of our method requires stronger conditions on the dimension of the minimal abelian variety for $Y$ and it 
leads, therefore, only to partial results which we are going to study in a work in progress \cite{CVV_progress}.

 The main result of  this paper is the following effective  theorem: 
\begin{thm}\label{codimensionerelativauno}
Let $V$ be a weak-transverse variety in a  product $E^N$ of elliptic curves with CM  defined over a number field $k$. Then the maximal $V$-torsion anomalous varieties $Y$ of relative codimension one are finitely many. In addition  the degrees and normalized heights are bounded as follows. For any positive real $\eta$, there exist  constants depending only on $E^N$ and $\eta$ such that:
\begin{enumerate}
\item\label{mainprima} if $Y$ is not a translate then
\begin{align*} h(Y) &\ll_\eta   (h(V)+ \deg V)^{\frac{N-1}{N-\dim V-1}+\eta},\\
\deg Y &\ll_\eta\deg V (h(V)+ \deg V)^{\frac{\dim V}{N-\dim V-1}+\eta};
\end{align*}
\item\label{mainseconda} if $Y$ is a translate of positive dimension then
\begin{align*} h(Y)&\ll_\eta {(h(V)+\deg V)}^{\frac{N-2}{N-\dim V-1}+\eta}{[k_\mathrm{tor}(V):k_\mathrm{tor}]}^{\frac{\dim V-1}{N-\dim V-1}+\eta},\\
\deg Y&\ll_\eta (\deg V){(\cV)}^{\frac{\dim V-1}{N-\dim V-1}+\eta}.\end{align*}
\item\label{mainterza} if $Y$ is a point then {its N\'eron-Tate height and its degree are bounded as}
\begin{align*}\hat{h}(Y)  &\ll_\eta (h(V)+\deg V)^{\frac{N-1}{N-\dim V-1}+\eta}[k_{\mathrm{tor}}(V):k_{\mathrm{tor}}]^{\frac{\dim V}{N-\dim V-1}+\eta},\\
[\qe(Y):\mathbb{Q}]
&\ll_\eta ({\deg V[k(V):k]})^{\frac{N-1}{N-\dim V -1}+\eta}{(\cV)}^{\frac{\dim V(N-1)}{(N-\dim V-1)^2}+\eta}.
\end{align*}
\end{enumerate}
\end{thm}
 The notations will be made precise in the next sections. {The proof of Theorem~\ref{codimensionerelativauno} is split in 3 sections: part~(1) is proved in  Section~\ref{seztanontraslati}, part~(2) in  Section~\ref{seztadimposi}, part~(3) in  Section~\ref{seztadimzero}.
 
An interesting aspect of the bounds in our result is that the dependence on $V$ is completely explicit and the constants depend only on the dimension, degree and height of the ambient variety and on $\eta$; this kind of uniformity was also implicit but not explicitly stated in \cite{BMZOnUnlikely}.

Those in Theorem \ref{codimensionerelativauno} are only some of the bounds we obtain. 
In addition we bound the degree of the torsion varieties  $B+\zeta$ minimal for the  maximal $V$-torsion anomalous that we consider; also in these bounds the implied constants are uniform in the variety $V$.

Our  method is completely effective. As mentioned above, it is well known that an effective proof of Theorem \ref{codimensionerelativauno} for any relative codimension implies the effective Mordell-Lang Conjecture, a problem outside the reach of the current methods and theories; so the restriction to  relative codimension one is strong, but natural.
In  Section \ref{sezML} we give some implications of our theorem on the effective and quantitative Mordell-Lang Conjecture.\\

Unlike in \cite{BMZ1}, the proof of our main theorem avoids  the use of Ax's result and  is based on the following main ingredients: the deep Zhang inequality; the strong explicit Arithmetic B\'ezout  Theorem by Philippon; a  sharper variant of an  effective result by Galateau on the Bogomolov  Conjecture for non-translates;  the relative Lehmer  bound in CM abelian varieties by Carrizosa for points; an inductive method for translates of positive dimension. The use of a Lehmer-type bound {forces} the assumption of CM. In Section \ref{ingredienti} we recall these mentioned results.

\medskip 

We now present some consequences and applications of our main theorem.  An immediate corollary is  Conjecture \ref{TORAN}  for weak-transverse varieties of codimension $2$, proving the  analogue of Theorem 1.7 of \cite{BMZ1} in a  product of CM elliptic curves. 
\begin{cor}\label{mainuno}
Let $V$ be a weak-transverse variety of codimension $2$ in a  product of elliptic curves with CM.
Then $V^{ta}$ is a dense open subset of $V$. 
\end{cor}
\begin{proof}
 It is sufficient to consider maximal $V$-torsion anomalous varieties. Let $Y$ be a  maximal $V$-torsion anomalous component of $V\cap(B+\zeta)$.

Then, by definition of $V$-torsion anomalous variety, 
$$ \codim Y< \codim V+ \codim B.$$
Equivalently 
$$\dim B-\dim Y<\codim V=2.$$ If  $Y$ has relative codimension zero, then $Y= B+\zeta$ and $Y$ is a component of the closure of the  torsion contained in $V$, which is a proper closed  subset  by the Manin-Mumford Conjecture. If $Y$ has relative codimension one, we apply our main theorem.
\end{proof}
Note that, in this particular case, Theorem~\ref{codimensionerelativauno} provides a bound for the normalized height and degree of the components of $V\setminus V^{ta}$ which are not torsion varieties. Bounds for the torsion components were already known in the literature.
\medskip

\medskip

\bigskip

Conjecture \ref{TORAN} is well known to be related to the CIT (Conjecture \ref{CIT}). We now discuss the implication of our result on the CIT.
By definition, for any  torsion variety  $B+\zeta$, the intersection $V^{ta}\cap (B+\zeta)$ has the right dimension. In particular,  if a point of $V$ lies in some algebraic subgroup of codimension $\geq \dim V +1$, then that point is contained in a $V$-torsion anomalous variety (we would expect empty intersection), and so it does not belong to $V^{ta}$.  Then,  as a further consequence  of our theorem, we obtain an effective version of the following:

\begin{cor}
\label{maindue}
Let $V$ be a  weak-transverse variety of codimension $2$ or a  weak-transverse translate in a   product of elliptic curves with CM. 
 Then $S_{\dim V+1}(V)$ is non-dense in $V$ and its closure is $V \setminus V^{ta}$.

\end{cor}
\begin{proof}
If $\codim (B+\zeta) \ge \dim V+1$, then all components of $V\cap (B+\zeta)$ are torsion anomalous so they do not intersect $V^{ta}$. Therefore  
\[
V^{ta }\cap  \bigcup_{\codim\, \acca \ge \dim V+1}\acca=\emptyset
\]
and $S_{\dim V+1}(V) \subseteq V\setminus V^{ta}.$ By  Proposition \ref{wtrans},  for weak-transverse translates $V^{ta}=V$ and by  Corollary \ref{mainuno}, $V^{ta}$ is an open dense set for $V$ of codimension $2$. 
  That  the closure of $S_{\dim V+1}(V)$ is $V \setminus V^{ta}$ is proven exactly as in \cite{BMZ1}, page 25, for tori. Recall that,  with their definition, points are not torsion anomalous varieties.

\end{proof}
\medskip
Clearly the bounds for the {normalized height and the degree of the} components of $V \setminus V^{ta}$ are the  same as in Theorem \ref{codimensionerelativauno}. As an immediate corollary we obtain:
\begin{cor}
Let $C$ be a weak-transverse curve  in $E^3$ where $E$ is an elliptic curve with CM defined over a number field $k$.   Then
$$S_2(C)=C\setminus C^{ta}$$ is a finite set of cardinality and N\'eron-Tate height effectively bounded. 
 In particular every non-torsion point $Y_0\in S_2(C)$ satisfies
\begin{align*}
 \hat{h}(Y_0) &\ll_\eta (h(C)+\deg C)^{2+\eta}[k_{\mathrm{tor}}(C):k_{\mathrm{tor}}]^{1+\eta},\\
 [k (Y_0):\qe] &\ll_\eta (\deg C (h(C)+\deg C)[k_{\mathrm{tor}}(C):k_{\mathrm{tor}}][k(C):k])^{2+\eta}.
\end{align*}
\end{cor}

This result points out the novelty of our main theorem. 
It is known since 2003 that $S_2(C)$ is a finite set for weak-transverse curves in powers of CM elliptic curves (see Theorem~1.7 in \cite{RemondViada}). Later on, it was shown that the CM assumption can be removed through the use of a Bogomolov-type bound  instead of a Lehmer-type bound (see \cite{ioant}). However, these results are non-effective as both rely on a non-effective use of the Voj\-ta height inequality to show that the height is bounded from above. This approach was also adapted to the case of weak-transverse curves in multiplicative tori to prove similar results (see \cite{Maurin}). With an independent argument, based on the use of an effective Mordell-Lang Theorem in tori, Bombieri, Habegger, Masser and Zannier \cite{BHMZ} give an
 effective bound for the height of $S_2(C)$ and consequently for its cardinality, for weak-transverse curves in tori.  The first effective bound for the height of $S_2(C)$ is given for transverse curves in a product of elliptic curves in \cite{ioannali}.  So this  corollary is a first example of an effective bound  for the height for  a weak-transverse curve  in abelian varieties.

{In higher  dimensions,  the Bounded Height Conjecture  proved by Habegger (\cite{philipp}, theorem on page~407), together with the Effective Bogomolov Bound by Galateau (\cite{galateau}, Theorem~1.1) and  the {Non-density Theorem} by Viada (\cite{ijnt}, Theorem~1.6)  imply the CIT for varieties  with $V^{oa}$ non-empty embedded in  certain abelian varieties which include all CM abelian varieties.}

An even stronger result involving a subgroup of finite rank $\Gamma$ of $A(\bar{\qe})$ was proved by R{\'e}mond in his series of articles \cite{RemondI}, \cite{RemondII} and \cite{RemondIII}, using a different approach based on the theory of Vojta height inequalities (see Theorem~1.3 of the last article for the non-density statement).

The original formulation of the Bounded Height Theorem in \cite{hab} did not
contain an explicit height bound and the author did not discuss the effectivity of the result (however, in a forthcoming paper, Habegger provides an explicit version of the Bounded Height Theorem for tori). Even if the method  is made effective, his assumption  $V^{oa}\not=\emptyset$ is  stronger than transversality. In this respect, ours is a new effective method in the context of the CIT for weak-transverse varieties. 

 Many are the contributions on the CIT of several authors. For a more extensive list, we refer to the references given in the papers mentioned above. \medskip

\section {Applications to the   Mordell-Lang Conjecture}\label{sezML}\label{effML}
The CIT is well known to have implications on the Mordell-Lang Conjecture.  
 The toric case of this last conjecture has been extensively studied, also in its effective form, by many authors.  An effective version for the toric case can be found in \cite{BiGi}, Theorem 5.4.5, and generalisations in \cite{BGEP}. However a general effective result in abelian varieties is not known.

From our Theorem \ref{codimensionerelativauno}, we obtain a bound for the height of the set of points belonging to a curve $C$ and a group $\Gamma$ of rank one as $\mathrm{End}(E)$-module. In addition we bound the cardinality of $C\cap \Gamma$, improving in some cases the completely general bound by R\'emond 
 (\cite{remond}, Theorem 1.2).
Our bounds are explicit in the  dependence on $C$ and $\Gamma$. 
Our method is  different from the two classical effective methods known in abelian varieties: the method by Chabauty and Coleman, surveyed, for example, in \cite{MP} and the method by Manin and Demjanenko, described in Chapter~5.2 of \cite{serreMWThm}, which are more general but of difficult application.

\medskip

Let $E$ be an elliptic curve defined over the algebraic numbers.

We let $\hat{h}$ be the standard N\'eron-Tate height on $E^N$; if $V$ is a subvariety of $E^N$, we shall denote by $h(V)$ the normalized height of $V$, as defined in \cite{patriceI}.  The height of a set is as usual the supremum of the heights of its points. If $E$ is defined over a field $k$, we  denote by $k_\mathrm{tor}$ the field of definition of all torsion points of $E$. All the constants in the following theorems become explicit if the constant for the Lehmer  type bound of Carrizosa in \cite{carrizosaIMRN} is made explicit.

\begin{cor}
\label{ML1}
Let $C$ be a weak-transverse curve in  $E^N$, with $E$ a CM elliptic curve and $N>2$. Let $k$ be a field of definition for $E$. Let $\Gamma$ be a subgroup of $E^N$ such that the group  generated by its coordinates  is an ${\mathrm{End}}(E)$-module of   rank one. Then, for any positive $\eta$, there {exist constants} $\cuno,\cdue$, depending only on $E^N$ and $\eta$, such that the N\'eron-Tate height of $C\cap \Gamma $ is bounded by
\begin{equation*}
\cuno (h(C)+\deg C)^{\frac{N-1}{N-2}+\eta}[k_{\mathrm{tor}}(C):k_{\mathrm{tor}}]^{\frac{1}{N-2}+\eta}
\end{equation*}
and the number of non-torsion points in $C\cap \Gamma $  is bounded  by
\begin{align*}
 &\cdue  {(\cC)}^{\frac{(N-1)(4N^2-N-4)}{2(N-2)^2}+\eta}\\
&\cdot  (\deg C)^{\frac{2N^3-N^2+N-4}{2(N-2)}+\eta}[k(C):k]^{\frac{N(N-1)(2N+1)}{2(N-2)}+\eta}.
\end{align*}

We also recall that by a result of R{\'e}mond in \cite{remond}, the number of torsion points in $$C\cap \Gamma $$ is bounded by 
\begin{equation*}
 \cremond (\deg C)^{2N^2},
\end{equation*}
where $\cremond$ is a constant depending only on $E^N$ and the fixed embedding.

\end{cor}

\begin{proof}
Let $g$ be an element of infinite order in $\overline\Gamma$, the group generated by all coordinates of any element in $\Gamma$.  Because $\overline\Gamma$ has rank 1, if a point $x=(x_1,\dots , x_N)$ is in $\Gamma$ then there exist $0\neq a_i,b_i \in{\mathrm{End}}(E)$ and  torsion points $\zeta_i\in E$ such that $$a_ix_i=b_i g+\zeta_i, \qquad i=1,\dotsc,N.$$ If all $b_i=0$ then $x$ is a torsion point, thus it has height zero. Otherwise, we can eliminate $g$ {from the $N$ equations above, leaving us with $N-1$ equations} in the $x_i$ which define  a torsion variety $H$ of codimension $N-1$ in $E^N$. Thus, $$ (C\cap \Gamma)\subseteq (C\cap \cup_{\dim H=1}H) \cup (C\cap {\mathrm{Tor}}_{E^N})=S_{N-1}(C)$$ for $H$ ranging over all algebraic subgroup of dimension one. However, if $N>2$ and $C$ is weak-transverse,  any point $x$ in the intersection 
$C\cap H$ is a maximal $C$-torsion anomalous point and, if  $x$ is not torsion it has relative codimension one in $H$. Applying {Theorem \ref{codimensionerelativauno}} we deduce the estimate on the height. The estimate on the cardinality for the points which are not torsion now follows from the bounds in Theorem \ref{tadimzero2}, while the number of torsion points is bounded by Theorem 2.1 in \cite{remond}. 
\end{proof}

If $N=2$, the intersection $C \cap H$ is not torsion anomalous, so we must follow another line. For  this reason, we need the assumption of $C$ being transverse.

\begin{cor}\label{MLdue}
Let $C$ be a transverse curve in  $E^2$ with $E$ a CM  elliptic curve  defined over a number field $k$. Let $\Gamma$ be a subgroup of $E^2$ such that the group $\overline\Gamma$ of its coordinates  is an ${\mathrm{End}}(E)$-module of   rank one,  and let $g\in\overline\Gamma $ be an element of infinite order. Then, for any positive $\eta$ there {exist constants $\ctre,\cquattro$,} depending only on $E^N$ and $\eta$, such that {the N\'eron-Tate height and the cardinality of the set $$C\cap \Gamma $$  are bounded  as}
\begin{align*}
\hat{h}(C\cap \Gamma)  &\leq \ctre[k_\mathrm{tor}(C\times g):k_\mathrm{tor}]^{1+\eta}(h(C)+(\hat{h}(g)+1)\deg C)^{2+\eta}\\
 \sharp(C\cap \Gamma) &\leq \cquattro {([k_\mathrm{tor}(C\times g):k_\mathrm{tor}](h(C)+(\hat{h}(g)+1)\deg C))}^{29+\eta}(\deg C)^{22+\eta}\\
 &\cdot [k(C\times g):k]^{21+\eta}.
\end{align*}
\end{cor}
\begin{proof}
Consider the curve $C'=C\times g$ in $E^3$. Then $C'$ is weak-transverse and  
  $C\cap \Gamma$ is embedded  in $C'\cap \cup_{\dim H=1}H$ with   $\hat{h}(C\cap \Gamma)\leq \hat{h}(C'\cap \cup_{\dim H=1}H)$, just as in the proof of Corollary \ref{ML1}.
However any point $x$ in the intersection $C'\cap H$ is a maximal $C'$-torsion anomalous point of relative codimension   one in $H$.
The bound on the height is obtained by Theorem \ref{codimensionerelativauno} applied to $C'\subseteq E^3$, with $\deg C'=\deg C$ and $$h(C')\leq 2\mu(C')\deg C'=2\mu(C)\deg C+2\hat{h}(g)\deg C\leq 2( h(C)+\hat{h}(g)\deg C),$$ which follows applying both sides of Zhang's inequality, recalled in Section \ref{sezzhang} below. The bound on the cardinality follows from the bounds in Theorem \ref{tadimzero2}.
\end{proof}

The assumptions on the curve are necessary in both corollaries. Indeed for a weak-transverse curve $C$ in $E^2$ the above theorem is not true. Consider the weak-transverse curve $C=E\times g$ with $g$ non-torsion;  for any positive integer $m$, let $\Gamma_m$ be generated by the point $\gamma_m=(mg,g)$. For any $m$, the point $\gamma_m$  belongs to $C$ and, for $m$ which goes to infinity, the height of $\gamma_m$ tends to infinity. Thus there cannot be general bounds independent of $\Gamma$.
This does not happen in higher codimension. The analogue would be $C=E\times g\times g'$ where $g$ and $g'$ are linearly independent to ensure the weak-transversality. But  $C\cap \Gamma$ is empty if $\overline\Gamma$ has rank one.

We also remark that for $C$ weak-transverse, $C\times g$ is not necessarily weak-transverse: let $C=E\times g$ with $g$ non-torsion, then $E\times g\times g$ is contained in the abelian subvariety $x_2=x_3$.

\section{Torsion anomalous varieties: preliminary results }

 In this section we denote by $G$ an abelian variety or a torus and by $\mathrm{Tor}_G$ the set of all torsion points in $G$. In the next statements we consider subvarieties $V$ of $G$.

\subsection{Maximality  and minimality}

To show that the union of all $V$-torsion anomalous varieties is non-dense, we only need to consider maximal ones.
 We recall from the introduction the following definition.
\begin{D} We say that a  $V$-torsion anomalous  variety $Y$ is \emph{maximal}  if it is not contained in any $V$-torsion anomalous  variety of strictly larger dimension. \end{D}
On the other hand, a variety $Y$ can be a component of the intersection of $V$ with different  torsion varieties. We want to choose a  minimal torsion variety which makes $Y$  anomalous. 

 \begin{D} Let $Y$ be a  $V$-torsion anomalous variety.  We say that the irreducible  torsion variety $B+\zeta$ is \emph{minimal} for $Y$ if   $Y$ is an irreducible component of $V\cap (B+\zeta)$, $$\codim Y<\codim V+\codim B,$$ and $B+\zeta$ has minimal dimension among the irreducible  torsion varieties with such properties. \end{D}
 Note that  there is a unique minimal torsion variety for $Y$, and it is contained in every torsion variety containing $Y$.
Indeed if $B+\zeta$ is minimal for $Y$, and $B'+\zeta'$ is any other torsion variety containing $Y$, then also $B''+\zeta''=(B+\zeta)\cap (B'+\zeta')$ is a torsion variety such that $Y$ is a component of $V\cap (B''+\zeta'')$, and $\codim Y < \codim V +  \codim  B''$. Therefore by minimality $\dim(B+\zeta)= \dim(B''+\zeta'')$. But $(B+\zeta)$ is irreducible, thus $B+\zeta=B''+\zeta''$, and therefore $B+\zeta\subseteq B'+\zeta'$.

\subsection{Conjecture~\ref{TORAN} implies the Torsion Openness Conjecture and the Torsion Finiteness Conjecture} 

Here we prove, as stated in the introduction, that Conjecture~\ref{TORAN} implies the Torsion Openness Conjecture and the Torsion Finiteness Conjecture, by showing that it is equivalent to the finiteness of the maximal $V$-torsion
anomalous  varieties.

\begin{propo}\label{EquivalenzaCongetture}
Conjecture~\ref{TORAN} implies that there are finitely many maximal $V$-torsion
anomalous  varieties of any dimension.
\end{propo}

\begin{proof}

Clearly, it is sufficient to prove that every component of $V\setminus V^{ta}$ is $V$-torsion anomalous.

We remark that if $Y \subseteq V$ is a $V$-torsion anomalous subvariety, and $Y\subseteq X \subseteq V$ for some subvariety $X$ of $V$, then $Y$ is $X$-torsion anomalous, because $\cod X\geq\cod V$.

Let now $X$ be a component of $V\setminus V^{ta}$, and let $Y_i$ be the family of maximal $V$-torsion anomalous subvarieties contained in $X$; let us write $B_i+\zeta_i$ for the torsion variety which is minimal for $Y_i$. We shall prove that $X$ is $V$-torsion anomalous.

Let $X_1,\dotsc,X_r$ be the other components of $V\setminus V^{ta}$. Because the $Y_i$ cover $X\setminus\left(\bigcup_{j=1}^r X_j \right)$, if follows that $X^{ta}$ is not dense, and therefore Conjecture~\ref{TORAN} implies that $X$ is contained in a torsion variety; let $H+\xi$ be the minimal such variety.
Notice that all $B_i+\zeta_i$ are contained inside $H+\xi$, otherwise $(B_i+\zeta_i)\cap (H+\xi)\subseteq B_i+\zeta_i$ would be minimal for $Y_i$.

Consider now the abelian variety $H$, and the weak transverse subvariety $X-\xi$. By Conjecture~\ref{TORAN} applied to this new setting, not all the $Y_i-\xi$ can be $(X-\xi)$-torsion anomalous in $H$. This means that for a certain index $i_0$, 
\begin{equation}\label{equivalenza1}
\cod_H (X-\xi)=\dim(B_{i_0}+\zeta_{i_0}-\xi)-\dim(Y_{i_0}-\xi).
\end{equation}
But $Y_{i_0}$ is $V$-torsion anomalous, so
\begin{equation}\label{equivalenza2}
\dim (B_{i_0}+\zeta_{i_0})-\dim Y_{i_0}<\cod V
\end{equation}
Combining \eqref{equivalenza1} and \eqref{equivalenza2} we easily obtain 
\[
\dim(H+\xi)-\dim X < \cod V.
\]
Let now $Z$ be a component of $V\cap (H+\xi)$ containing $X$. Clearly \[\dim(H+\xi)-\dim Z\leq \dim(H+\xi)-\dim X < \cod V,\]
therefore $Z$ is $V$-torsion anomalous and thus contained in $V\setminus V^{ta}$. This implies that $X=Z$, completing the proof.
\end{proof}

\subsection{Relative position of the torsion anomalous varieties}\label{relposTAvar} 

Without loss of generality we can work with a maximal $V$-torsion anomalous $Y$ and its minimal torsion variety $B+\zeta$.
The maximality for $Y$ avoids redundancy and the minimality assures the weak-transversality of $Y$ in $B+\zeta$, as defined below.
The relative position of a $V$-torsion anomalous variety $Y$ in $B+\zeta$  is  determinant and leads to the following natural definition.

\begin{D} An irreducible variety $Y$  is \emph{weak-transverse}  in a torsion variety $B+\zeta$ if  $Y\subseteq (B+\zeta)$  and $Y$ is not contained in any proper torsion subvariety of $B+\zeta$. Similarly $Y$  is \emph{transverse}  in a translate $B+p$ if $Y\subseteq (B+p)$  is not contained in any translate strictly contained in $B+p$. The codimension of $Y$ in $B+\zeta$ is called the \emph{relative codimension of $Y$ in $B+\zeta$}; we simply  say the \emph{relative codimension of $Y$} if $Y$ is $V$-torsion anomalous and $B+\zeta$ is minimal for $Y$. 
\end{D}

Then, we have the following lemma.
 
\begin{lem} 
\label{wt}
Let $Y$ be a maximal $V$-torsion anomalous variety and  let $B+\zeta$  be minimal for $Y$. Then $Y$ is weak-transverse in $B+\zeta$.\end{lem}
\begin{proof}
 Assume that  $Y$ is not weak-transverse in $B+\zeta$, then it is contained in an irreducible  torsion subvariety $B'+\zeta'$  of $B+\zeta$ with $\codim (B'+\zeta')>\codim (B+\zeta)$. So $Y$ is a component of  $ V \cap (B'+\zeta')$. In addition
 $$\codim Y <\codim V+\codim (B+\zeta)<\codim V+\codim (B'+\zeta'),$$
which contradicts the minimality of  $ (B+\zeta)$. 
 \end{proof}

\subsection{Torsion anomalous varieties as components of different intersections}
In the next lemma we prove that every $V$-torsion anomalous variety which is a component of $V\cap (B+\zeta)$, is also a component of any intersection $V\cap (A+\zeta')$  with $B+\zeta\subseteq A+\zeta'$ {and $\dim A$ not too big}.  We remark that we can choose $\zeta=\zeta'$.
\begin{lem} 
\label{cruciale}
Let $Y$ be a maximal $V$-torsion anomalous variety, and let $B+\zeta$ be minimal for $Y$. Then $Y$ is a component of $V\cap (A+\zeta)$ for every algebraic subgroup $A\supseteq B$, with $\codim A\ge \dim V-\dim Y$. 
\end{lem}
\begin{proof}
Clearly $Y\subseteq V\cap (A+\zeta)$. Let $X$ be  an irreducible component of $V\cap (A+\zeta)$ which contains $Y$, then
  \begin{equation*}
  \codim X \le \codim V+\codim A.
  \end{equation*}

  If the inequality is strict then  $X$ is anomalous and by the maximality of $Y$ we get $Y=X$.
  
  Otherwise, since by assumption $\codim A\ge \dim V-\dim Y$, we have
  $$\codim X = \codim V+\codim A\ge \codim V +\dim V-\dim Y=\codim Y,$$
but $Y\subseteq X$, so the opposite inequality  also holds.
This implies $\codim X=\codim Y$, so $Y=X$ and $Y$ is a component of $V\cap (A+\zeta)$.
 \end{proof}

\subsection{No torsion anomalous varieties on a weak-transverse translate}
\label{23}
The simple choice of maximal and minimal varieties and the group structure allow us to prove Conjecture \ref{TORAN} for weak-transverse translates. 
\begin{propo}\label{wtrans} Let $H+p$ be a weak-transverse translate in an abelian variety. Then the set of $(H+p)$-torsion anomalous varieties is empty.
\end{propo}
\begin{proof}
Let $Y$ be an $(H+p)$-torsion anomalous variety and let  $B+\zeta$ be minimal for $Y$. Then $Y$ is a component of $(H+p)\cap (B+\zeta)$ and $$\dim H+\dim B-\dim Y<N.$$
Remark that whenever $(H+p)\cap (B+\zeta)\neq \emptyset$, then $p=h+b+\zeta$ for some $h\in H$ and $b\in B$. Thus $$(H+p)\cap (B+\zeta)=(H+b+\zeta)\cap (B+\zeta)=(H+b+\zeta)\cap (B+b+\zeta)=(H\cap B)+b+\zeta$$ and $$\dim(H\cap B)=\dim((H+p) \cap (B+\zeta)).$$

In our case  $Y\subseteq (H+p)\cap (B+\zeta)$, then $$\dim(H\cap B)\geq \dim Y.$$ Thus
$$\dim(H+B)=\dim H +\dim B -\dim(H\cap B)\le \dim H +\dim B -\dim Y   <N,$$
where the last inequality is obtained from the fact that $Y$ is $V$-torsion anomalous. 
Therefore $H+B$ is a proper algebraic subgroup of the ambient variety. Since $p\in H+B+\zeta$, then $H+p\subseteq H+B+\zeta$, against the weak-transversality of $H+p$.
\end{proof}
\subsection{Finitely many maximal $V$-torsion anomalous varieties in $B+\mathrm{Tor}_G$}
 Let $V$ be a weak-transverse variety in $G$.
  Let us fix an irreducible torsion subvariety $B$ of $G$; we end this section by showing the
 finiteness of the maximal $V$-torsion anomalous varieties in $ V\cap (\sotto+{\mathrm{Tor}}_{G})$ for which $B+\zeta$ is minimal, for some torsion point $\zeta$. 
 This result implies the finiteness of the maximal $V$-torsion anomalous varieties, if one can uniformly bound the degree of the  corresponding minimal torsion variety $B+\zeta$.

We first prove that the maximal $V$-torsion anomalous components of intersections of the form $B+\zeta$, for $\zeta\in\mathrm{Tor}_{G}$, are  non-dense in $V$. 
\begin{lem}\label{sottogruppo1} 
Let $V$ be a weak-transverse variety in $G$.
Let $\sotto$ be an abelian subvariety of $G$.  Then the  union of all $V$-torsion anomalous varieties which are components of an intersection of the form $V\cap(\sotto+\zeta)$, for $\zeta\in\mathrm{Tor}_{G}$, is non-dense in $V$.

\end{lem}
\begin{proof}
Suppose that there exists a dense  union of maximal $V$-torsion anomalous varieties which are components of $V\cap(B+\zeta)$, for $\zeta$ ranging over all torsion points of $G$. By the  box principle, there exists a dense subunion of $d$-dimensional $V$-torsion anomalous varieties  $Y_i\subseteq V\cap (\sotto+\zeta_i)$, for $\zeta_i$ torsion points, where $d$ is the maximal integer having this property. 

Consider the natural projection $\pi_\sotto: G\to G/\sotto$. As the $Y_i$ are dense and project to  points, the fibre dimension theorem tells us that 
$$\dim V = \dim Y_i + \dim \pi_\sotto(V).$$
Note that $V$ is not $V$-torsion anomalous, because $V$ is weak-transverse, thus $\dim Y_i<\dim V$ and $\dim \pi_\sotto(V)$ is at least one.
However the $Y_i$ are anomalous therefore
$$\codim Y_i <\codim V + \codim\sotto.$$
We deduce that
$$\dim \pi_\sotto(V) < \codim\sotto=\dim G/\sotto.$$
This shows that ${\pi_\sotto}_{|V}$ is not surjective on $G/\sotto$. Since $V$ is weak-transverse,  also $\pi_\sotto(V)$ is weak-transverse  in $G/\sotto$. Notice  that the image  via $\pi$ of any $Y_i$ is a torsion point. By  the Manin-Mumford Conjecture the closure of the torsion of $\pi_\sotto(V)$ is non-dense, thus also its preimage is non-dense in $V$. This contradicts the density of the $Y_i$. \end{proof}

An application of this result actually shows the finiteness of the maximal $V$-torsion anomalous varieties in $ V\cap (\sotto+{\mathrm{Tor}}_{G})$  for which $B+\zeta$ is minimal, for some torsion point $\zeta$.

\begin{lem}\label{sottogruppo}
Let $V$ be weak-transverse  in $G$.
 Let $\sotto$ be an abelian subvariety of $G$. Then there exist only finitely many torsion points $\zeta$ such that $V\cap(\sotto+\zeta)$ has a maximal $V$-torsion anomalous component for which $B+\zeta$ is minimal.

\end{lem}

\begin{proof}
 Let $\dim G=N$. From Lemma \ref{sottogruppo1} the closure of all maximal $V$-torsion anomalous $Y_i$ which are components of  an intersection of the form $V\cap(B+\zeta)$, for $\zeta\in\mathrm{Tor}_{G}$, is a proper closed subset of $V$. In particular, if we consider only those for which $B+\zeta$ is minimal, their closure is still a proper closed subset.  We write it as the finite union of its irreducible components $X_1 \cup \dots \cup X_R$. We now show that the $X_j$ 
are exactly   the maximal $V$-torsion anomalous varieties such that  $B+\zeta$ is minimal, for some $\zeta\in \mathrm{Tor}_{G}$.  Suppose  that  $X_1$ is not  a maximal $V$-torsion anomalous component for which $B+\zeta$ is minimal, with $\zeta \in \mathrm{Tor}_{G}$. Then the $Y_i$ contained in $X_1$ are dense in $X_1$, $\dim X_1> \dim Y_i$ and $X_1$ is not $V$-torsion anomalous due to the maximality of the $Y_i$. 

 By assumption,  $B+\zeta_i$ is minimal for $Y_i$ for some $\zeta_i\in {\mathrm{Tor}_{G}}$.
Then $X_1$ is not contained in any $B+\zeta$, otherwise $Y_i\subseteq X_1\subseteq V\cap (B+\zeta)$ and $X_1$ would be a $V$-torsion anomalous variety; this gives a contradiction.

 Let $H+\zeta$ be the torsion variety (not necessarily proper) of smallest dimension containing $X_1$. Then $X_1$ is weak-transverse in $H+\zeta$, and since $X_1$ is not $V$-torsion anomalous 
\begin{equation}
\label{finB}N-\dim X_1=N-\dim V+N-\dim H.\end{equation}
 Recall that, the $Y_i$ are $V$-torsion anomalous varieties and  use \eqref{finB}  to obtain
$$N-\dim Y_i<N-\dim V+N-\dim B=\dim H-\dim X_1+N-\dim B.$$
whence 
$$ \dim H-\dim Y_i<\dim H-\dim X_1+\dim H-\dim B.$$
 We now claim that $(B+\zeta_i)\subseteq(H+\zeta)$. Indeed if it were not so, the intersection $(B+\zeta_i)\cap(H+\zeta)$ would contradict the minimality of $B+\zeta_i$ for $Y_i$. Therefore $(B+\zeta_i-\zeta)\subseteq H$; translating  every variety by $\zeta$, we obtain that $X_1-\zeta$ is weak-transverse in $H$ and the  $Y_i-\zeta$ are dense in $X_1-\zeta$ and 
$(X_1-\zeta)$-torsion anomalous. This contradicts Lemma \ref{sottogruppo1} applied to $X_1-\zeta$ in $H$.\end{proof}

\section{Main ingredients}\label{ingredienti}

\subsection{Notation}\label{seznotaz}
 Recall that all varieties are assumed to be defined over the field of algebraic numbers.  
Let $A$ be an abelian variety; to  a symmetric  ample line bundle $\elle$ on $A$ we  attach an embedding $i_\elle: A\hookrightarrow \mathbb{P}^m$ defined by the minimal power of $\elle$ which is very ample. Heights and degrees corresponding to $\elle$ are computed via such an embedding. More precisely, the degree of a subvariety of $A$ is the degree of its image under $i_\elle$;  $\hat{h}=\hat{h}_\elle$  is the $\elle$-canonical N\'eron-Tate height of a point in $A$, and $h$ is the normalized height of a subvariety of $A$ as defined, for instance, in \cite{patriceI}.

 Most often we consider products of an elliptic curve $E$. Then, we denote by $\cani_1$ the line bundle on $E$ defined by the neutral element, and by $\cnbundle$ the bundle on $E^N$ obtained as the tensor product of the pull-backs of $\cani_1$ through the $N$ natural projections. 

Unless otherwise specified, we compute degrees and heights on $E^N$ with respect to $\cnbundle$. 

\medskip

We note  that by a result of Masser and W\"ustholz in \cite{Masserwustholz} Lemma 2.2, every abelian subvariety of $E^N$ is defined over a finite extension of $k$ of degree bounded by $3^{16 N^4}$. For this reason we always assume that all abelian subvarieties are defined over $k$. Up to a field extension of degree two, we also assume that every endomorphism of $E$ is defined over $k$.

\medskip

By $\ll$ we {always denote} an inequality up to a multiplicative constant depending only  on $E$ and $N$.

\subsection{Subgroups and torsion varieties}
\label{torsione}

Let $B+\zeta$ be an irreducible torsion variety of $E^N$ with $\codim\sotto=r$.
We  associate  $\sotto$ with a morphism $\varphi_\sotto:E^N \to E^{r}$ such that $\ker \varphi_\sotto=\sotto+\tau$ with $\tau$ a torsion set of absolutely bounded cardinality (by \cite{Masserwustholz} Lemma 1.3).  
In turn $\varphi_\sotto$ is identified with a matrix in $\mathrm{Mat}_{r\times N}(\mathrm{End}(E))$ of rank $r$ such that the degree of   $\sotto$ is essentially (up to constants depending only on $N$) the sum of  the squares of the determinants of the minors of $\varphi_\sotto$. By {Minkowski}'s theorem,  such a sum is essentially the product of the  squares $d_i$ of the norms  of the  rows of the matrix representing $ \varphi_\sotto$ (see for instance \cite{FuntorialeAV} for more details).

In short  $B+\zeta$ is a component of the torsion variety given as the  zero set of forms $h_1,   \dots , h_r$,  which are the rows of $\varphi_\sotto$,  of degree    $d_i$. In addition  $$d_1  \cdots d_r \ll \deg (B+\zeta)\ll d_1  \cdots d_r  .$$ 
We  {assume} to have ordered the $h_i$  by increasing degree. 

\medskip

We also recall that, as is well known, we can use Siegel's lemma to complete the matrix defining $B$ to a square invertible matrix; this gives a construction for the orthogonal complement $B^\perp$ and shows that $\sharp(B\cap B^\perp)\ll (\deg B)^2$.

\medskip

As remarked in \cite{ioant}, in a product of different elliptic curves an algebraic subgroup is associated with a matrix where the entries corresponding to the non-isogenous factors are all zero. 

For this reason our theorems, which we prove in $E^N$ for simplicity,   hold in products of different elliptic curves as well.

\subsection{The Zhang Estimate}\label{sezzhang}
We recall the following definition.
\begin{D}\label{defiessmin}
For a variety $V\subseteq A$, the \emph{essential minimum $\mu(V)$} is the supremum of the reals $\theta$ such that the
set  $ \{ x \in V(\overline{\qe}) \mid \hat{h} (x)\le \theta\}$
is non-dense in $V$.
\end{D}
The Bogomolov Conjecture, proved by Ullmo and Zhang in 1998, asserts that 
the essential minimum $\mu(Y)$ is strictly positive if and only if $Y$ is non-torsion.

From the crucial result in Zhang's proof of the Bogomolov Conjecture (see \cite{ZhangEquidistribution}) and from the definition of normalized height, we have that for an irreducible subvariety $X$ of an abelian variety:
\begin{equation}\label{zhang}
\mu(X) \le \frac{h(X)}{\deg X} \le (1+\dim X) \mu(X).
\end{equation}
\subsection{The Arithmetic B\'ezout theorem}
The following version of the Arithmetic B\'ezout theorem is due to Philippon \cite{patrice}.

\begin{thm}[Philippon]\label{bezout}
Let $X$ and $Y$ be irreducible subvarieties of the  projective space $\mathbb{P}^n$ defined over $\overline{\mathbb{Q}}$. Let $Z_1, \dots , Z_g$ be the irreducible components of $X\cap Y$. Then 
$$ \sum_{i=1}^g h(Z_i)\le \deg X h(Y) +\deg Y h(X) +c(n) \deg X \deg Y,$$
where $c(n)$ is a constant depending only on $n$.
\end{thm}

\subsection{An effective Bogomolov Estimate  for relative transverse varieties}

The following theorem is a sharp effective version of the Bogomolov Conjecture for weak-transverse varieties.  It is an elliptic analogue, up to a lower order term, of a toric conjecture of Amoroso and David in \cite{fra}.

\begin{thm}[Checcoli-Veneziano-Viada \cite{FuntorialeAV}] \label{due}
 Let $E^N$ be a  product of elliptic curves, and let $Y$ be an irreducible  subvariety of $E^N$  transverse in a translate $\sotto +p$. Then, for any $\eta>0$, there exists a positive constant $\ccinque$  depending on $E^N$ and $\eta$, such that
$$\mu (Y) \ge \ccinque  \frac{
(\deg \sotto)^{\frac{1}{\dim B-\dim Y}-\eta}
}{
(\deg Y)^{{\frac{1}{\dim B-\dim Y}}+ \eta}
}.$$
  \end{thm} 

The dependence on the  degree of $B$ is crucial in our applications. This  theorem holds for $E$ with or without CM and it is a special case of the main theorem of \cite{FuntorialeAV}.
Its proof is based on an equivalence between  line bundles, and on  the  lower bound  for the essential minimum of a transverse variety with respect to the standard bundle $\cnbundle$. Such a bound is given by Galateau in \cite{galateau}.    The constant $\ccinque$ is effective and becomes explicit if the constants in  \cite{galateau} are made explicit\footnote{In a personal communication Galateau provided us explicit computations.}.

\subsection{A relative Lehmer Estimate for points}\label{sezcarri}

In \cite{carrizosaIMRN}, Theorem 1.15, Carrizosa proves the so called relative Lehmer problem for CM abelian varieties.
Her theorem  generalises to CM abelian varieties a result of Ratazzi in \cite{ratazziIMRN} for one CM elliptic curve.  The proof of Ratazzi is inspired  to the  theorem of Amoroso and Zannier in \cite{AZ} for algebraic numbers.
As a straightforward corollary of her theorem we have the following:

\begin{thm}[Carrizzosa]\label{carri2}
Let $E$ be an elliptic curve with CM defined over a field $k$. Let $P$ be a point of infinite order in $E^N$, and let $B+\zeta$ be the torsion variety of minimal dimension containing $P$, with $B$ an abelian subvariety and $\zeta$ a torsion point. Then for every $\eta>0$ there exists a positive constant $\csei$ depending on $E^N$ and $\eta$, such that
\[
\hat{h}(P)\geq \csei  \frac{(\deg B)^{\frac{1}{\dim B}-\eta}}{[k_\mathrm{tor}(P):k_\mathrm{tor}]^{\frac{1}{\dim B}+ \eta}}.
\]
\end{thm}

The effectivity in the relative Lehmer is not explicitly stated in the theorem of Carrizosa.\footnote{In a short personal communication she claims that her constants are effective.} Using also the effectivity of other results, as for instance the result of Amoroso and Zannier in \cite{AZ}\footnote{Though in \cite{AZ} the authors are not concerned with effectivity, their result can be made effective (see \cite{ADelsinne}).} and of David and Hindry \cite{davidhindry}, one may check that Carrizosa's constant can be made effective. In addition, the complicated descent in her article can be replaced by the simple induction argument presented for tori in \cite{amoviadacomm}. An analogous effective method for the relative Lehmer in tori is given by Delsinne, see \cite{Delsinne}.

\section{Torsion anomalous varieties which are  not translates}\label{seztanontraslati}
Let $V$ be a weak-transverse variety in a   power of elliptic curves.
In this section we prove the finiteness of the maximal $V$-torsion anomalous varieties which are not translates and have relative codimension one, thus establishing part (\ref{mainprima}) of Theorem \ref{codimensionerelativauno}. Note that this  part of our theorem holds in any   power of elliptic curves, independently if  it has or not CM. 

\begin{thm}
\label{weakstrict}
Let $V\subseteq E^N$ be a weak-transverse variety. Then the   maximal   $V$-torsion anomalous varieties of relative codimension one which are  not translates  are finitely many. More precisely, let $Y$ be a maximal $V$-torsion anomalous variety which is not a translate. Assume that $Y$ has  relative codimension one in its minimal $B+\zeta$. Then for any $\eta>0$ there exist  constants depending only on $E^N$ and $\eta$ such that: 
\begin{itemize}
\item[] $$ \deg B \ll_\eta   (h(V)+ \deg V)^{\frac{\codim B}{\codim V -1}+\eta}, $$
\item[] $$h(Y)  \ll_\eta   (h(V)+ \deg V)^{\frac{\codim B}{\codim V-1}+\eta} $$ and
\item[]  $$ \deg Y \ll_\eta \deg V (h(V)+ \deg V)^{\frac{\codim B}{\codim V-1} -1+\eta}.$$
 \end{itemize}
 In addition  the torsion points $\zeta$ belong to a  finite set.
 
\end{thm} 

\begin{proof}
Let $Y$ be a maximal  $V$-torsion anomalous  variety which is  not a translate. Let $B+\zeta$ be minimal for $Y$. Then $Y$ is a component  of $V\cap (B+\zeta)$ with $B$ an abelian variety and $\zeta$ a torsion point. In addition $\codim Y<\codim V+\codim B$.

We prove that $\deg B$ is bounded only in terms of  $V$ and $E^N$; we then deduce the bounds  for ${h}(Y)$ and $\deg Y$.
 By assumption $Y$ is not a translate, and it has codimension one in $B+\zeta$; therefore $Y$ is transverse in $B+\zeta$. As points are translates and $V$ is not contained in a torsion variety, we have $1\leq \dim Y<\dim V.$ 
Applying the Bogomolov estimate Theorem \ref{due} to $Y$ in $B+\zeta$ we get
\begin{equation}\label{sotto}\frac{(\deg B)^{1-\eta}}{(\deg Y)^{{1}+\eta}}\ll_\eta \mu(Y).\end{equation}

We set $r=\codim B$. Let $h_1, \dots , h_r$ be the forms of increasing degrees $d_i$ such that   $B+\zeta$ is a component of their zero set,  as recalled  in Section \ref{torsione}. Then 
\begin{equation}\label{gradi}
 d_1 \cdots d_r \ll \deg (B+\zeta)=\deg B  \ll  d_1 \cdots d_r.
\end{equation}

 We denote $$r_1=\dim V-\dim Y.$$ Note that $r_1< r$, because $Y$ is $V$-torsion anomalous.
Let $A$ be the algebraic subgroup given by the first $h_1\cdots h_{r_1}$ forms. 
Then $\deg A \ll d_1 \cdots d_{r_1} .$
 Let $A_0$ be an irreducible component of $A$ containing $B+\zeta$.
 We remind that we have ordered the $h_i$ in such a way that the sequence $d_1,\dotsc,d_r$ is increasing, thus the successive geometric means are increasing too, which gives, together with \eqref{gradi}, the bound
 $$\deg A_0  \ll d_1 \cdots d_{r_1}\ll (d_1\dotsm d_r)^\frac{r_1}{r}\ll (\deg B)^\frac{r_1}{r}.$$ 
We also have $\codim A_0=r_1=\dim V-\dim Y$. 
By Lemma \ref{cruciale},  $Y$ is a component of $V\cap A_0$. We apply the Arithmetic B\'ezout theorem to $V\cap A_0$ and recall that $h(A_0)=0$, because $A_0$  is a torsion variety. Then 
\begin{equation}\label{sopra}h(Y) \ll (h(V)+ \deg V) \deg A_0\ll \cVv (\deg B)^{\frac{r_1}{r}}.\end{equation} 
For the irreducible variety $Y$ of positive dimension, 
Zhang's inequality \eqref{zhang} says
$$\mu(Y)\leq \frac{h(Y)}{\deg Y}.$$
Combining this with \eqref{sopra} and \eqref{sotto} we obtain
$$\frac{{(\deg B)}^{1-\eta}}{(\deg Y)^{1+\eta}}\ll_\eta\mu(Y)\ll  (h(V)+ \deg V)\frac{ (\deg B)^\frac{r_1}{r}}{\deg Y}.$$

Recall that $Y$ is a component of $ V\cap (\sotto+\zeta)$. By B\'ezout's theorem $\deg Y \le \deg \sotto  \deg V$.  Thus 
\[{(\deg B)}^{1-\eta}\ll_\eta  (h(V)+\deg V)(\deg B)^\frac{r_1}{r}{(\deg \sotto \deg V)}^{\eta}\] and therefore \[{(\deg B)^{\frac{r-r_1}{r}-2\eta}}\ll_\eta  (h(V)+\deg V){ (\deg V)}^{\eta}.\]

 Since $r-r_1=\codim V -1$, for $\eta$ small enough  we get
 \begin{equation}
 \label{forte}
 \deg B \ll_\eta   (h(V)+ \deg V)^{\frac{r}{\codim V -1}+\eta}(\deg V)^{\eta} .
 \end{equation}

So we  have proved that the degree of $B$ is bounded only in terms of $V$ and $E^N$. 
\\ Since the  abelian subvarieties of bounded degree are finitely many, applying Lemma \ref{sottogruppo}  we conclude that $\zeta$ belongs to a finite set. \\Finally, the bound on the  height of $Y$ is given by  (\ref{sopra}) and (\ref{forte}) 
$$h(Y)\ll_\eta \cVv^{\frac{r}{\codim V -1}+\eta}(\deg V)^{\eta}.$$
The bound on the degree is given by  B\'ezout's theorem for the component $Y$ of $V\cap A_0$ and (\ref{forte}) \[\deg Y\ll_\eta \cVv^{\frac{r}{\codim V -1}-1+\eta}(\deg V)^{1+\eta}.\qedhere\]
 \end{proof}

\section{Torsion anomalous points}\label{seztadimzero}   In this and the following section we prove  that, if $V$ is a weak-transverse variety in a   power of elliptic curves  with CM, then the maximal $V$-torsion anomalous varieties {of relative codimension one} which are  translates are non-dense in $V$. {We  now prove this statement in the case of $V$-torsion anomalous varieties of dimension zero, thus establishing part (\ref{mainterza}) of Theorem \ref{codimensionerelativauno}}. The proof relies on the Arithmetic B\'ezout theorem, the Zhang's inequality and on the relative Lehmer, Theorem \ref{carri2}. {Recall that, as} the last bound is proved only for CM elliptic curves we need this assumption.

\begin{thm}
\label{tadimzero}
Let $V\subseteq E^N$ be a  weak-transverse variety, where $E$ has CM. Then, the set of maximal  $V$-torsion anomalous points of relative codimension one is a finite set of explicitly bounded {N\'eron-Tate} height and relative degree. 

More precisely, let $k$ be a field of definition for $E$ and let $k_\mathrm{tor}$ be the field of definition of all torsion points of $E^N$.  Let $d$ be the dimension of $V$. Let $Y_0$ be a maximal $V$-torsion anomalous point and let $B+\zeta$ be minimal for $Y_0$, with $\dim B=1$. Then
\begin{align*}
 \deg B &\ll_\eta (\cV)^{\frac{N-1}{N-1-d}+\eta} ,\\
\hat{h}(Y_0)  &\ll_\eta (h(V)+\deg V)^{\frac{N-1}{N-1-d}+\eta}[k_{\mathrm{tor}}(V):k_{\mathrm{tor}}]^{\frac{d}{N-1-d}+\eta},\\
[k_{\mathrm{tor}}(Y_0):k_{\mathrm{tor}}] &\ll_\eta \deg V[k_\mathrm{tor}(V):k_\mathrm{tor}]^{\frac{N-1}{N-1-d}+\eta}{(h(V)+\deg V)}^{\frac{d}{N-1-d}+\eta}.
\end{align*}
In addition the torsion points $\zeta$  have order effectively bounded in Theorem \ref{tadimzero2}.
\end{thm}

\begin{proof}
Let $Y_0$ be a  maximal  $V$-torsion anomalous point, with $\sotto+\zeta$ minimal for $Y_0$.

By assumption $\dim B=\codim_{B+\zeta}Y_0=1$. We proceed to bound $\deg B$ and, in turn, the height of $Y_0$ and its degree over $k_\mathrm{tor}$.
 To this aim we shall use Theorem \ref{carri2} and the Arithmetic B\'ezout theorem.

By Section \ref{torsione}, the variety  $\sotto+\torsione$ is a component of the  torsion variety defined as the zero set of  forms $h_1,\dots,h_{N-1}$ of increasing degrees $d_i$ and $$d_1 \cdots d_{N-1}\ll\deg \sotto=\deg (\sotto+\torsione)\ll d_1 \cdots d_{N-1}.$$
Consider the torsion variety defined as the zero set of the first  $d$ forms $h_1, \dots, h_d$, and take  a connected component $A_0$ containing $\sotto+\torsione$.  Then 
\begin{equation}\label{degA0}
\deg A_0\ll d_1 \cdots d_d\ll{(\deg B)}^{\frac{d}{N-1}}
\end{equation} and $$\codim A_0=d=\dim V-\dim Y_0.$$

By Lemma \ref{cruciale},  each component of $V\cap (B+\zeta)$  is a component of $V\cap A_0$.     
All conjugates of $Y_0$ over $k_\mathrm{tor}(V)$ are in $V\cap (B+\zeta)$, so the number of components of $V\cap A_0$ of height $\hat{h}(Y_0)$ is at least $$[k_\mathrm{tor}(V,Y_0):k_\mathrm{tor}(V)]\geq \frac{[k_\mathrm{tor}(Y_0):k_\mathrm{tor}]}{[k_\mathrm{tor}(V):k_\mathrm{tor}]}.$$

We then apply the Arithmetic B\'ezout theorem to $V\cap A_0$ obtaining
\begin{equation}\label{arbez61}
[k_\mathrm{tor}(Y_0):k_\mathrm{tor}]\hat{h}(Y_0)\ll {\cV}(\deg B)^{\frac{d}{N-1}}.
\end{equation}

 Applying Theorem \ref{carri2} to $Y_0$ in $B+\zeta$, we obtain that for every positive real $\eta$

\begin{equation}\label{car61}
\hat{h}(Y_0)\gg_\eta \frac{(\deg B)^{1-\eta}}{[k_{\mathrm{tor}}(Y_0):k_{\mathrm{tor}}]^{1+\eta}}.
\end{equation}

Combining \eqref{car61} and \eqref{arbez61}
we have
\begin{align*}
\frac{(\deg B)^{1-\eta}}{[k_{\mathrm{tor}}(Y_0):k_{\mathrm{tor}}]^{\eta}}&\ll_\eta[k_{\mathrm{tor}}(Y_0):k_{\mathrm{tor}}]\hat{h}(Y_0)\ll\\
&\ll {\cV}(\deg B)^{\frac{d}{N-1}}.
\end{align*}
For $\eta$ small enough we obtain
\begin{equation}\label{degB61}
\deg B\ll_\eta {(\cV)}^{\frac{N-1}{N-1-d}+\eta}[k_{\mathrm{tor}}(Y_0):k_{\mathrm{tor}}]^\eta.
\end{equation}
 Apply now B\'ezout's theorem to $V\cap A_0$. All the conjugates of $Y_0$ over $k_\mathrm{tor}(V)$ are components of the intersection, so\begin{equation}\label{boundktor61}
\frac{[k_{\mathrm{tor}}(Y_0):k_{\mathrm{tor}}]}{[k_{\mathrm{tor}}(V):k_{\mathrm{tor}}]}\ll_\eta {(\cV)}^{\frac{d}{N-1-d}+\eta} (\deg V)^{1+\eta},
\end{equation}
which gives the last bound in the statement. Substituting \eqref{boundktor61} back into \eqref{degB61} we have the bound on $\deg B$.

Finally apply  the Arithmetic B\'ezout theorem to $V\cap A_0$ to get
\begin{equation}\label{hY061}
\hat{h}(Y_0)\ll (h(V)+\deg V)(\deg B)^{\frac{d}{N-1}}\ll_\eta(h(V)+\deg V)^{\frac{N-1}{N-1-d}+\eta}[k_{\mathrm{tor}}(V):k_{\mathrm{tor}}]^{\frac{d}{N-1-d}+\eta}.
\end{equation}

 Having bounded $\deg B$, in view of Lemma \ref{sottogruppo}, the points $\zeta$ belong to a finite set.
\end{proof} 
Notice that in Theorem \ref{tadimzero} we have effectively bounded the degree of the abelian variety $B$, and we  applied Lemma \ref{sottogruppo} to prove the finiteness of the points $Y_0$ in a non effective way. In the following theorem we explicitly bound the degree $[k(Y_0):\qe]$ for $Y_0$ not a torsion point. This, together with  the bound for $\hat{h}(Y_0)$ in Theorem \ref{tadimzero}, allows to effectively find all $V$-torsion anomalous points of relative codimension one. 

\begin{thm}\label{tadimzero2}
Let $V$ be a  weak-transverse variety  in $E^N$, where $E$ has CM. Let $k$ be a field of definition for $E$. Let $d$ be the dimension of $V$. Let $Y_0$ be a maximal $V$-torsion anomalous point and let $B+\zeta$ be minimal for $Y_0$ with $\dim B=1$. Then
$$[k(Y_0):\mathbb{Q}]\ll_\eta {(\cV)}^{\frac{d(N-1)}{(N-1-d)^2}+\eta}{(\deg V[k(V):k])}^{\frac{N-1}{N-1-d}+\eta}.$$

In addition  the torsion points $\zeta$ can be chosen with  $$[k(\zeta):\mathbb{Q}]\ll [k(Y_0):\mathbb{Q}]$$ and order  bounded by 
$$\mathrm{ord}(\zeta)\ll_\eta [k(Y_0):\mathbb{Q}]^{\frac{N}{2}+\eta}.$$ 

Finally let $S$ be the number of maximal $V$-torsion anomalous points of relative codimension one. Then
$$S\ll_\eta {(\cV)}^{A_1+\eta}{(\deg V)}^{A_2+1+\eta}[k(V):k]^{A_2+\eta}$$
where
\begin{align*}
A_1&= \frac{(N-1)(2(N+1)(N-d-1)+dN(2N+1))}{2(N-d-1)^2} \leq  (N+1)^4,\\
A_2&= \frac{N(N-1)(2N+1)}{2(N-d-1)} \leq N^3.
 \end{align*}
\end{thm}
\begin{proof}
In view of Theorem \ref{tadimzero}, we know that $\deg B$ and $\hat{h}(Y_0)$ are bounded.
We now proceed to bound $[k(Y_0):k]$. 

To this aim we need to construct an algebraic subgroup $\sbgrpG$ of codimension $d$ defined over $k$, containing $Y_0$ and of controlled degree. In order to do this we use Siegel's lemma in a similar way as in \cite{ioannali}, Proposition 3, which in turn follows the work \cite{BMZ_tori} of Bombieri, Masser and Zannier in tori. Here we use Siegel's lemma directly on equations with coefficients in the endomorphism ring of $E$.

We notice that $\mathrm{End}(E)$  is an order in an imaginary quadratic field $L$ with ring of integers $\mathcal{O}$.
 The coordinates of $Y_0=(x_1,\ldots,x_N)$ generate an $\mathcal{O}$-module $\Gamma$ of rank one. The torsion submodule of $\Gamma$ is  well known; for instance in  \cite{ioannali}, Proposition 2  we find a description for it. Such a torsion module is  clearly  $\mathcal{O}$ invariant. As a $\mathbb{Z}$-module  it is generated by two points $T,\tau T$ of exact orders $R,R'$ respectively,  where $R=c(\tau)R'$ and $c(\tau)$ is essentially the real part of $\tau^2$, so a constant of the problem. 

Therefore  we can write
$$x_i=\alpha_i g +\beta_i T$$
for a fixed point of infinite order $g$ in $E$, with coefficients $\alpha_i,\beta_i\in\mathcal{O}$ such that 
\begin{equation}\label{ai}
\hat{h}(x_i)=|N_L(\alpha_i)|\hat{h}(g)
\end{equation}
and $$N_L(\beta_i)\ll R^2.$$

\medskip

We want to find coefficients $a_i\in\mathcal{O}$ such that $\sum_i^N a_i x_i=0$.
This gives a linear system of 2 equations, obtained equating to zero the coefficients of $g$ and $T$. The system has coefficients in $\mathcal{O}$ and $N+1$ unknowns: the $a_i$'s and one more unknown for the congruence relation arising from the torsion point.

We use a version of Siegel's lemma over $\mathcal{O}$ as stated in \cite{BiGi}, Section 2.9, to get $d$ equations with coefficients in $\mathcal{O}$; multiplying them by a constant depending only on $E$, we may assume that they have coefficients in $\mathrm{End}(E)$. Thus they define the sought-for algebraic subgroup $\sbgrpG$ of degree
$$\deg \sbgrpG\ll \left((\max_i N_L(\alpha_i))(\max_i N_L(\beta_i)) \right)^{\frac{d}{N-1}}.$$
Let $\sbgrpG_0$ be a $k$-irreducible component of $\sbgrpG$ passing through $Y_0$.
Then $$\deg \sbgrpG_0\ll \left((\max_i N_L(\alpha_i)) R^2 \right)^\frac{d}{N-1}.$$

By the maximality of $Y_0$, the point  $Y_0$ is a component of $V\cap \sbgrpG_0$  and by B\'ezout's theorem we get  $$\frac{[k(Y_0):k]}{[k(V):k]}\leq \deg V\deg \sbgrpG_0.$$
Using also \eqref{ai}, we obtain
\begin{alignat}{2}
\frac{[k(Y_0):k]}{[k(V):k]}&\ll \deg V \left((\max_i N_L(\alpha_i)) R^2 \right)^\frac{d}{N-1}\leq \label{erreesse}\\
&\leq \deg V \left(R^2\frac{ \max_i \hat{h}(x_i)}{\hat{h}({g})}\right)^\frac{d}{N-1}.\notag
\end{alignat}
Notice that $\hat{h}(x_i)\leq\hat{h}(Y_0)$ for all $i$.

We can now apply Theorem \ref{carri2} to $g$ in $E$, obtaining for every $\eta>0$
\begin{equation}\label{ratazzi61}
\frac{1}{\hat{h}(g)}\ll_\eta {[k_\mathrm{tor}(g):k_\mathrm{tor}]^{1+\eta}}\leq [k_\mathrm{tor}(Y_0):k_\mathrm{tor}]^{1+\eta}
\end{equation}
because $g$ is defined over $k(Y_0)$.

The product $[k_\mathrm{tor}(Y_0):k_\mathrm{tor}]\hat{h}(Y_0)$ was bounded in \eqref{arbez61}, so using also the bounds in Theorem \ref{tadimzero} we obtain
\begin{equation}\label{boundcombinato62}
 [k(Y_0):k]\ll_\eta \deg V [k(V):k](\cV)^{\frac{d}{N-1-d}+\eta}R^{2\frac{d}{N-1}}.
\end{equation}

 By a result of Serre in \cite{serre}, recalled also in \cite{ioannali}, Corollary 3,  for $R$ larger than a constant and  $\varphi$  the Euler function, we have $$\varphi(R)\varphi(R')\ll [k(Y_0):k] .$$ In addition $$R^{2-\eta}\ll_\eta\varphi(R)\varphi(R')$$ since in general $\varphi(x)\gg_\eta x^{1-\eta}$ and $R$ and $R'$ are related by a constant.
From this and \eqref{boundcombinato62}, for $\eta$ small enough we obtain
\begin{equation}\label{erreesse2}
[k(Y_0):k]\ll_\eta  \deg V [k(V):k]{(\cV)}^{\frac{d}{N-d-1}+\eta} [k(Y_0):k]^{\frac{d}{N-1}+\eta}.
\end{equation}
Since $Y_0$ is $V$-torsion anomalous of relative  codimension one, we have $d<N-1$ and we deduce 
\begin{align}
&[k(Y_0):\mathbb{Q}]\ll[k(Y_0):k] \ll_\eta    \notag\\
&\ll_\eta ({\deg V[k(V):k]})^{\frac{N-1}{N-d-1}+\eta}{(\cV)}^{\frac{d(N-1)}{(N-d-1)^2}+\eta}.\label{bY0}
\end{align}

 We now want to bound the degree of $\zeta$ over $k$ and the order of $\zeta$. 
 Let $K$ be the field of definition of $B+\zeta$. Notice that $K\subseteq k(\zeta)$: in fact if $\sigma\in \mathrm{Gal}(\overline{k}/k(\zeta))$, then $\sigma(B+\zeta)=B+\zeta$. We are going to prove that $[k(\zeta):K]$ is absolutely bounded and that $[K:k]\leq [k(Y_0):k]$. 

\medskip

From \cite{bertrand}, we can choose $\zeta$ in a complement $B'$ of $B$ such that $B\cap B'$ has cardinality bounded  only in terms of $N$. 
Now let $\sigma\in \mathrm{Gal}(\overline{k}/K)$ and suppose that $\sigma(\zeta)\neq \zeta$. Then $\sigma(B+\zeta)=B+\sigma(\zeta)=B+\zeta$,  because $B$ is defined over $k$. Since $\zeta, \sigma(\zeta)\in B'$ we have $\sigma(\zeta)-\zeta\in B\cap B'$.  So $[k(\zeta):K]\ll 1$.

\medskip

We also notice that $K\subseteq k(Y_0)$, otherwise we would have a $\sigma\in\Gal(\overline{k}/k)$ such that $\sigma(Y_0)=Y_0$, but $\sigma(B+\zeta)\neq B+\zeta$. If this were the case, $Y_0$ would be a component of $V\cap(B+\zeta)\cap \sigma(B+\zeta)$, against the minimality of $B+\zeta$.

 Thus $$[k(\zeta):k]=[k(\zeta):K][K:k]\ll [K:k]\leq [k(Y_0):k].$$
In view of \eqref{bY0},  $\zeta$ generates an extension of $k$ of bounded degree. By Serre's result mentioned above 
\begin{equation}\label{ordz}
\textrm{ord}(\zeta)\ll_\eta [k(Y_0):\mathbb{Q}]^{\frac{N}{2}+\eta}.
\end{equation}

We are left to give an explicit bound for the number of maximal $V$-torsion anomalous points $Y_0$ of relative codimension one. This is obtained in the following way: we first bound the number of possible subgroups $B$ and possible torsion points $\zeta$ such that $B+\zeta$ is minimal for some $Y_0$. Then we apply B\'ezout's theorem to every intersection $V\cap (B+\zeta)$. 

We already proved in Theorem \ref{tadimzero} and \eqref{ordz} that if $B+\zeta$ is minimal for $Y_0$, then $\deg B$ and $\textrm{ord}(\zeta)$ are bounded.

By  Section \ref{torsione}, the number of  abelian subvarieties  $B$ in $E^N$ of dimension one and degree at most $\deg B$ is $\ll_\eta({\deg B})^{N+\eta}$, for every $\eta>0$.  In fact, if $B$ is such a abelian subvariety, consider its associated matrix, as recalled in Section \ref{torsione}; it is an $(N-1)\times N$ matrix and we call $d_i$  the square of the norm of its $i$-th row. Then the number of possible choices for the elements $d_i$ is bounded by   $\delta(\deg B)^{N-1}$, where, for a positive integer $n$, $\delta(n)$ counts the number of divisors of $n$.
We notice that, for every $\eta>0$ we have $$ \delta(\deg B)\ll_\eta (\deg B)^{\eta}.$$
Now, for every choice of the $d_i$, the number $D$ of $(N-1)\times N$ matrices in which  the square of the norm of the $i$-th row is at most $d_i$ is bounded in the following way $$D\ll \left(\prod_{i=1}^{N-1} d_i\right)^N\ll ({\deg B})^N.$$

So for every $\eta>0$, the number of possible subgroups $B$ is $\ll_\eta(\deg B)^{N+\eta}.$

As for the point $\zeta$, it is well known that the number of torsion points in $E^N$ of order bounded by a constant $M$ is at most $M^{2N+1}$. In fact the number of points of order dividing a positive integer $i$ is $i^{2N}$; so a bound for the number of torsion points of order at most $M$  is given by $$\sum_{i=1}^M i^{2N}\ll M^{2N+1}.$$

Applying B\'ezout's theorem to every intersection $V\cap (B+\zeta)$, we obtain that for every $\eta>0$ the number $S$ of $V$-torsion anomalous points  of relative codimension one is bounded by $$S\ll_\eta \deg V (\deg B)^{N+1+\eta}  {\textrm{ord}(\zeta)}^{2N+1}.$$
This, combined with Theorem \ref{tadimzero}, \eqref{ordz} and \eqref{bY0}, gives the required explicit bounds.
\end{proof}

\section{Torsion anomalous translates of positive dimension}\label{seztadimposi}

We let $V$ be a weak-transverse variety in a   power of elliptic curves  with CM. In this section we study $V$-torsion anomalous varieties which are translates of positive dimension. We  reduce this case to the zero dimensional case.

First we compare the $V$-torsion anomalous translates with translates contained in $V$.
\begin{lem}\label{maximal_trans}
Let $V$ be a weak-transverse subvariety of an abelian variety  of dimension $N$. 
Let $Y$ be a maximal $V$-torsion anomalous translate, then  $Y$ is a maximal translate contained in $V$ (i.e. $Y$ is not  strictly contained in any translate contained in $V$).
\end{lem}
\begin{proof}

Let $Y$ be a maximal $V$-torsion anomalous translate and suppose that $Y$ is  contained in a maximal translate $(H+p)\subseteq V$ with $\dim(H+p)> \dim Y $.
Let $B+\zeta $ be minimal for $Y$. Then $Y$ is a component of $V\cap (B+\zeta)$  and \begin{equation}\label{disuno}\codim Y <\codim V +\codim B .\end{equation}

Since $(H+p)\cap (B+\zeta)\supseteq Y$ then $p=h+b+\zeta$ for some $h\in H$ and $b\in B$. 
Therefore $$(H+p)\cap (B+\zeta)=(H+b+\zeta)\cap (B+\zeta)=(H+b+\zeta)\cap (B+b+\zeta)=(H\cap B)+b+\zeta$$ and $$\dim(H\cap B)=\dim((H+p) \cap (B+\zeta))\ge \dim Y.$$
By (\ref{disuno}),  we deduce  \begin{equation}\begin{split}\label{distre}\dim(H+B)&=\dim H +\dim B -\dim(H\cap B)\\
 & \le \dim H +\dim B -\dim Y <N.\end{split}\end{equation}
So $H+B+ \zeta$ is a proper torsion subvariety of the ambient variety.
Moreover $$H+p=H+h+b+\zeta \subseteq H+B+\zeta.$$
 Thus $$H+p\subseteq V\cap (H+B+\zeta).$$ 
 
 By (\ref{disuno}) and  (\ref{distre}), we deduce 
 \begin{equation*}\begin{split}N-\dim H &< N-\dim V + N -\dim H- \dim B+\dim Y\\
 &\leq N-\dim V + N -\dim H- \dim B+\dim (B\cap H)\\&=N-\dim V + N -\dim (H+B).\end{split}\end{equation*}
 This implies that $H+p$ is a $V$-torsion anomalous translate, against the maximality of $Y$.
\end{proof}

The following lemma is {due to Patrice Philippon and it  relates the essential minimum of a translate to the height of the point of translation. More in general, for abelian varieties properties of orthogonality in the Mordell-Weil groups have been studied by D. Bertrand \cite{BertOrto}.}
\begin{lem}[Philippon]\label{minimoess}
Let $H+Y_0$ be a weak-transverse translate in  $E^N$, with  $Y_0$ a point in  the orthogonal complement $H^\perp$ of $H$. Then $$ \mu(Y_0)= \mu(H+Y_0).$$

\end{lem}
\begin{proof}
The points $Y_0+\zeta$, for $\zeta\in {\mathrm{Tor}}_{H}$, are dense in $H+Y_0$ and they have height equal to $\hat{h} (Y_0)$. So we get $\mu(H+Y_0)\le \mu(Y_0).$

To obtain the other inequality, consider  a set of points of the form ${x_i+Y_0}$ with $x_i \in H$, which is dense in $H+Y_0$. 

{Let 
\begin{equation*}
 \langle p,q\rangle_{NT}=\sum_{i=1}^N\frac{1}{2}\left( \hat{h}(p_i+q_i)-\hat{h}(p_i)-\hat{h}(q_i)\right)\in\mathbb{R}
\end{equation*} 
be the N\'eron-Tate pairing on $E(\overline{\qe})^N$, with $p=(p_1,\dotsc,p_N)$ and $q=(q_1,\dotsc,q_N)$ in $E(\overline{\qe})^N$.
In \cite{PhilipponLemmaPreprint},  P.~Philippon  proves that, given two {connected algebraic} subgroups $H'$ and $H''$ of $E^N$, then $H'(\overline{\qe})$ and $H''(\overline{\qe})$ are orthogonal, with respect to the N\'eron-Tate pairing, if and only if their tangent spaces at the origin are orthogonal in $\mathbb{C}^N$, with respect to the standard hermitian product. 

Using this result}  we get  $\hat{h} (x_i+Y_0)= \hat{h}(x_i) +\hat{h}(Y_0) \le \hat{h}(Y_0) $. \end{proof}

 We now state our main theorem for $V$-torsion anomalous translates, which proves part (\ref{mainseconda}) of Theorem \ref{codimensionerelativauno}. Let $H+Y_0$ be a maximal $V$-torsion anomalous translate of relative codimension one. The idea of the proof is to apply the functorial Lehmer-type bound by Carrizosa to the point $Y_0$ in the complement $H^\perp$ of $H$, so that the problem becomes zero dimensional. We then apply the Arithmetic B\'ezout theorem to $H+Y_0$ in the usual way. The link between $\mu(Y_0)$ and $\mu(H+Y_0)$ is then given by  Lemma \ref{minimoess}.
 Notice that when the translate is a point, the first two bounds in the statement reduce to the ones in Theorem \ref{tadimzero}, while the third is trivial. For this reason the theorem, as stated, partially generalises Theorem \ref{tadimzero}; its interest, however, lies in the case of translates of positive dimension, and the proof only deals with those, as the case of points is already proved by Theorem \ref{tadimzero}.

\begin{thm}
\label{trasla}
Let $V$ be a weak-transverse subvariety of $E^N$, where $E$ has CM. Then the set of  maximal $V$-torsion anomalous translates of relative codimension one is a finite set of explicitly bounded normalized height and degree.

More precisely, let $k$ be a field of definition for $E$  and let $k_\mathrm{tor}$ be the field of definition of all torsion points of $E^N$.
Let  $H$ be an abelian subvariety of $E^N$ and let $p$ be a point such that $H+p$ is  a maximal $V$-torsion anomalous translate of relative  codimension one. Let  $B+\zeta$ be  minimal for  $H+p$.  Then, for every real positive $\eta$  there exist  constants depending only on $E^N$ and $\eta$, such that 
\begin{align*}
 \deg B&\ll_\eta {(\cV)}^{\frac{\codim B}{\codim V-1}+\eta},\\
h(H+p)&\ll_\eta {(h(V)+\deg V)}^{\frac{\codim B}{\codim V-1}+\eta}{[k_\mathrm{tor}(V):k_\mathrm{tor}]}^{\frac{\dim V-\dim B +1}{\codim V-1}+\eta},\\
\deg (H+p)&\ll_\eta (\deg V){(\cV)}^{\frac{\dim V-\dim B +1}{\codim V-1}+\eta}.
\end{align*}

In addition the points $\zeta$ belong to a finite set (of cardinality absolutely bounded in Theorem \ref{trasla2}).
\end{thm}

\begin{proof}
 If $\dim(H+p)=0$, then the theorem follows from Theorem \ref{tadimzero}; from now on we assume that $\dim H >0$. 

As we remarked in Section \ref{sezcarri}, we can assume all abelian subvarieties of $E^N$ to be defined over $k$.

Let $Y_0$ be a point in the  orthogonal complement $H^\perp$ of $H$,  such that $H+p=H+Y_0$.  By Lemma \ref{minimoess},  
\begin{equation}
\label{A}\mu(Y_0)=\mu(H+Y_0).
\end{equation} 

We are going to use the Arithmetic B\'ezout theorem to find an upper bound for $\mu(H+Y_0)$ and Theorem \ref{carri2} to find a lower bound.

Let $B+\zeta $ be minimal for $H+Y_0$.  Then $H+Y_0$ is a component of $V\cap (B+\zeta)$  and by assumption $\codim_{B+\zeta} (H+Y_0) =1$.  By Lemma \ref{wt} $H+Y_0$ is weak-transverse in $B+\zeta$, thus $Y_0$ is not a torsion point and the coordinates of $Y_0$ generate a module of rank $\dim(H^\perp\cap(B+\zeta))=\codim_{B+\zeta} (H+Y_0) =1$.

 Let $r$ be the codimension of $B$. The variety  $\sotto+\torsione$ is a component of the torsion variety given by the zero set of  the forms $h_1,\dots,h_{r}$ of increasing degrees $d_i$ and $$d_1 \cdots d_{r}\ll\deg \sotto=\deg (\sotto+\torsione)\ll d_1 \cdots d_{r}.$$ 
Note that $H+Y_0$ is contained in $V$ and, by assumption, it has relative codimension one. Thus $\codim V<\codim H=r+1$.  Denote $$r_1= r+1-\codim V= \dim V-\dim H.$$ 

Consider the torsion variety defined as the zero set of the first $r_1$ forms $h_1,\dots,h_{r_1}$, and take  one of its irreducible component $A_0$ passing through $H+Y_0$. Then
\begin{equation}\label{degA0t}
\deg A_0\ll  {(\deg B)}^{\frac{r_1}{r}}
\end{equation} and $$\codim A_0=r_1.$$

Recall that $k_\mathrm{tor}$ is the field of definition of all torsion points in $E^N$.
By Lemma \ref{cruciale},  $H+Y_0$   is a component of $V\cap A_0$. 
Since the intersection $V\cap A_0$ is defined over $k_\mathrm{tor}(V)$, every conjugate of $H+Y_0$ over $k_\mathrm{tor}(V)$  is a component of $V\cap A_0$ and all such components have the same normalized height.

 So the number of components of $V\cap A_0$  of height $h(H+Y_0)$  is at least $$ \frac{[k_\mathrm{tor}(H+Y_0):k_\mathrm{tor}]}{[k_\mathrm{tor}(V):k_\mathrm{tor}]}.$$

 We apply the Arithmetic B\'ezout theorem to  $V\cap A_0$ and we obtain 
\begin{equation}\label{arbez742}
{h}(H+Y_0)\frac{[k_\mathrm{tor}(H+Y_0):k_\mathrm{tor}]}{[k_\mathrm{tor}(V):k_\mathrm{tor}]}\ll (h(V)+\deg V) {(\deg B)}^{\frac{r_1}{r}}.
\end{equation}
By the Zhang's inequality, \eqref{A} and \eqref{arbez742}, we deduce 
\begin{equation}\label{n7t}
\mu(Y_0)\ll \frac{\cV {(\deg B)}^{\frac{r_1}{r}}}{[k_\mathrm{tor}(H+Y_0):k_\mathrm{tor}]\deg H}.
\end{equation}

Consider the minimal torsion subvariety $H_0$ of $B$ containing $Y_0-\zeta$, so $Y_0 $ is not contained in any torsion subvariety of $H_0+\zeta$. The dimension of $H_0$ is then the rank of the subgroup generated by the coordinates of $Y_0-\zeta$, which is the same as the rank of the subgroup generated by $Y_0$, because $\zeta$ is a torsion point; as we have seen, this rank is 1.

Consider the irreducible component $\sbgrpH$  of the intersection $ H^\perp\cap (B+\zeta)$ containing $Y_0$: since $Y_0$ is not torsion, then $\sbgrpH$ has dimension one. So $H_0+\zeta= \sbgrpH$, because both varieties are irreducible, contain $Y_0$ and are one dimensional. 
Let now $H_2$ be the translate of $H_0$ passing through 0. Then, as $H_0+\zeta=\sbgrpH$ is the component of $H^\perp\cap (B+\zeta)$ passing through $Y_0$, we see that $H_2$ is the component of $H^\perp\cap B$ passing through 0, \textit{i. e.} the orthogonal abelian subvariety of $H$ in $B$. Hence $H+H_2=B$.

\medskip

We also notice that
\begin{equation}\label{t73conj}
 [k_\mathrm{tor}(Y_0):k_\mathrm{tor}]\leq [k_\mathrm{tor}(H+Y_0):k_\mathrm{tor}]\cdot \sharp(H\cap H_2).
\end{equation}
In fact if $\sigma\in\Gal(\overline{k_\mathrm{tor}}/k_\mathrm{tor}(H+Y_0))$ then $\sigma(Y_0)-Y_0\in H$. 
Since $H_2+Y_0=H_0+\zeta$ is defined over $k_\mathrm{tor}$, we have that $\sigma(Y_0)-Y_0\in H_2$ as well. Hence the number of conjugates of $Y_0$ over $k_\mathrm{tor}(H+Y_0)$ is at most $\sharp(H\cap H_2)$.

\medskip

For the lower bound for $\mu(Y_0)$, the proof follows  the case of dimension zero. 
In particular applying Theorem \ref{carri2} to $Y_0$ in $H_0+\zeta$ we get that, for every positive real $\eta$
\begin{equation}\label{carrizosa74}
\mu(Y_0)=\hat{h}(Y_0)\gg_\eta \frac{(\deg H_0)^{1-\eta}}{[k_\mathrm{tor}(Y_0):k_\mathrm{tor}]^{1+\eta}}.
\end{equation}

Combining \eqref{n7t}, \eqref{t73conj} and \eqref{carrizosa74} we get
\begin{equation}\label{t73combining}
\begin{split}
(\deg H_0)^{1-\eta}\ll_\eta {\cV} \frac{(\deg B)^{\frac{r_1}{r}}}{\deg H}\cdot\\ \cdot\sharp(H\cap H_2)^{1+\eta}[k_\mathrm{tor}(H+Y_0):k_\mathrm{tor}]^{\eta}.
\end{split}
\end{equation}

 Recall now that $H_2$ is the orthogonal of $H$ in $B$; therefore, applying Theorem~3 of \cite{BertrandDuality} to the decomposition $B=H+H_2$ we get
 \begin{equation*}
\deg H \deg H_2 \ll\sharp(H\cap H_2)\deg B  \ll \deg H \deg H_2;
\end{equation*}
$H_0$ and $H_2$ have, of course, the same degree, and therefore 
\begin{equation}\label{t73gradiHH'}
\deg H \deg H_0 \ll\sharp(H\cap H_2)\deg B  \ll \deg H \deg H_0.
\end{equation}

Thus, possibly changing $\eta$, we have 
\begin{equation}\label{bougradoB}
\deg B\ll_\eta  (\cV)^{\frac{r}{\codim V-1}+\eta} \left( \sharp(H\cap H_2)[k_\mathrm{tor}(H+Y_0):k_\mathrm{tor}]\right)^{\eta}.
\end{equation}

We now want to remove the dependence on $\sharp(H\cap H_2)$ and  $[k_\mathrm{tor}(H+Y_0):k_\mathrm{tor}]$ and bound the degree of the translate $H+Y_0$.

We first apply B\'ezout's theorem to the intersection $V\cap A_0$, obtaining: 
\begin{equation}\label{t73bez}
 \deg H\leq \deg V\deg A_0\ll \deg V(\deg B)^\frac{r_1}{r}
\end{equation}
and estimating the components we obtain
\begin{equation}\label{t73bez2}
 \deg H\frac{[k_\mathrm{tor}(H+Y_0):k_\mathrm{tor}]}{[k_\mathrm{tor}(V):k_\mathrm{tor}]}\leq \deg V\deg A_0\ll \deg V(\deg B)^\frac{r_1}{r}.
\end{equation}

By definition $H_2\subseteq H^\perp\cap (B+\zeta)$, therefore $H\cap H_2 \subseteq H\cap H^\perp$ and,
\begin{equation*}
\sharp(H\cap H_2)\leq \sharp(H\cap H^\perp)\ll (\deg H)^2.
\end{equation*}

This and   \eqref{t73bez2} give 
\[
 \sharp(H\cap H_2)[k_\mathrm{tor}(H+Y_0):k_\mathrm{tor}]\ll ([k_\mathrm{tor}(V):k_\mathrm{tor}]\deg V)^2 (\deg B)^{\frac{2r_1}{r}}.
\]

Substituting in \eqref{bougradoB} we get
\begin{equation}\label{bougradoB2} \deg B\ll_\eta {(\cV)}^{\frac{r}{\codim V-1}+\eta}.\end{equation}
 Then,  using  \eqref{t73bez} and replacing \eqref{bougradoB2}  we have
\begin{align*}
 \deg(H+Y_0) &\ll (\deg V)(\deg B)^\frac{r_1}{r}\ll_\eta\\
&\ll_\eta (\deg V){(\cV)}^{\frac{r_1}{\codim V-1}+\eta}.
\end{align*}

Finally, from \eqref{arbez742}  and \eqref{bougradoB2} we obtain
\begin{equation*}\label{stimah74}
h(H+Y_0)\ll_\eta {(h(V)+\deg V)}^{\frac{r}{\codim V-1}+\eta}{[k_\mathrm{tor}(V):k_\mathrm{tor}]}^{\frac{r_1}{\codim V-1}+\eta}.
\end{equation*}

Since we have bounded $\deg B$, we can conclude from Lemma \ref{sottogruppo} that the points $\zeta$ belong to a finite set.

\end{proof} 
In the proof we bounded the degree of $H+Y_0$ using B\'ezout's theorem, and we obtained a bound depending on $\deg V$ and $h(V)$. The dependence on $h(V)$ may be removed  with a different argument.

 Let $H+p\subseteq V$ be a translate which is maximal with respect to the inclusion among all such translates. Bombieri and Zannier in \cite{BZ}, Lemma 2, proved that only finitely many such abelian subvarieties $H$ can occur. More precisely the maximal  translates contained in $V$ have degree bounded only in terms of the degree and the dimension of $V$. As a corollary of their proof we obtain the following lemma.
\begin{lem}[Bombieri-Zannier]\label{Manin}
Let $V$ be a   weak-transverse variety. Then the maximal translates contained in $V$ are of the form $H+p$ for finitely many abelian subvarieties $H$ with $$\deg H \leq (\deg V)^{2^{\dim V }}.$$ 
\end{lem}
Using this lemma we obtain a bound which is more uniform, though the dependence on $\deg V$ shows a bigger exponent.

\bigskip

As remarked for the zero dimensional case, in  Theorem \ref{trasla} we proved finiteness  using the non effective Lemma \ref{sottogruppo}. We now give a completely effective result which is  the analogous of Theorem \ref{tadimzero2} in positive dimension. We bound the degrees of the fields of definition of the translates $H+p$ and of the torsion points $\zeta$.

\begin{thm}\label{trasla2}

Let $V$ be a weak-transverse subvariety of $E^N$, where $E$ has CM. Let $k$ be a field of definition for $E$.
 Let $H$ be an abelian subvariety of $E^N$, and let $p$ be a point such that $H+p$ is a maximal $V$-torsion anomalous translate of relative  codimension one; let  $B+\zeta$ be  minimal for  $H+p$.
Set $r=\codim B$; then the field $k(H+p)$ of definition of $H+p$  has degree bounded by 
\begin{multline*}
 [k(H+p):\mathbb{Q}]\ll_\eta [k(V):k]^{r+\eta} ({\deg V})^{3r-1}\cdot\\
\cdot (h(V)+\deg V)^{\frac{(2r-1)(r-\codim V+1)+r(r-1)}{\codim V-1}+\eta} {[k_\mathrm{tor}(V):k_\mathrm{tor}]}^{\frac{(3r-2)(r-\codim V +1)}{\codim V-1}}.
\end{multline*}

Moreover the torsion points $\zeta$  can be chosen so that
$$[k(\zeta):\mathbb{Q}]\ll_\eta  [k(H+p):\mathbb{Q}]$$ and $$\mathrm{ord}(\zeta)\ll_\eta [k(H+p):\mathbb{Q}]^{\frac{N}{2}+\eta}.$$
\end{thm}

\begin{proof}
We  keep all the notations and definitions used in Theorem \ref{trasla}.

Of course 
\begin{equation*}
[k(H+p):k]=[k(H+Y_0):k]\leq [k(Y_0):k],
\end{equation*}
because we are assuming all abelian subvarieties  of $E^N$ to be defined over $k$.

The bound on the degree $[k(Y_0):k]$ is obtained following the same idea of Theorem \ref{tadimzero2}: since $H_0$ has dimension one, the  group generated by the coordinates of $Y_0$ is an $\mathrm{End}(E)$-module of rank one.
We can apply Siegel's lemma in a similar way to the zero-dimensional case,  where we use 
 the estimate  \eqref{n7t}  for the height of $Y_0$.  In this case, however, we want to find an algebraic subgroup $\sbgrpGG$ of dimension 2, containing $Y_0$, and contained in $H^\perp$.

\bigskip

We know that $\mathrm{End}(E)$ is an order in an imaginary quadratic field $L$.
Let $\lambda_1,\dotsc, \lambda_{N-r-1}$ be linear forms which give equations for $H^\perp$, and define
\begin{equation*}
\delta_i=\max_j |N_L(l_{ij})|,
\end{equation*}
where $(l_{ij})_{j}$ is the vector of coefficients of the form $\lambda_i$. We now follow the steps of the proof of Theorem \ref{tadimzero2} and apply Siegel's lemma to obtain $r-1$ independent solutions which are also orthogonal to the vectors of coefficients of $\lambda_1,\dotsc, \lambda_{N-r-1}$; this time, in addition to the two equations of Theorem \ref{tadimzero}, we have also one equation for each of the $\lambda_i$, for a total of $N-r+1$ equations in $N+1$ unknowns ($N$ coefficients and one for the torsion point).

These $r-1$ vectors give $r-1$ linear forms which, together with the $\lambda_1,\dotsc, \lambda_{N-r-1}$, provide the $N-2$ equations needed to define $\sbgrpGG$. 

The bounds provided by Siegel's lemma give
\begin{align}
 \deg \sbgrpGG&\ll_\eta \left(\prod_{i=1}^{N-r-1}\delta_i \right)\left(\hat{h}(Y_0)[k(Y_0):k]^{1+\eta} \prod_{i=1}^{N-r-1}\delta_i\right)^\frac{r-1}{r}\ll_\eta\label{degG75}\\
&\ll_\eta (\deg H)^{2-\frac{1}{r}}\hat{h}(Y_0)^\frac{r-1}{r}[k(Y_0):k]^{\frac{r-1}{r}+\eta}.\notag
\end{align}

Now that we have found $\sbgrpGG$, we  first show that $H+Y_0$ is a component of $V\cap (H+\sbgrpGG)$:
indeed $\dim (H+\sbgrpGG)=\dim B +1$, because $\sbgrpGG\subseteq H^\perp$  has dimension 2; therefore any component of the intersection $V\cap (H+\sbgrpGG)$ with dimension greater than $\dim H$ is anomalous. But $H+Y_0$ is clearly contained in $V\cap (H+\sbgrpGG)$ and,  by the maximality of $H+Y_0$, it is not contained in any $V$-anomalous subvariety of greater dimension; hence it is itself a component of $V\cap (H+\sbgrpGG)$  and so are all its conjugates over $k(V)$.

Applying B\'ezout's theorem to $V\cap (H+\sbgrpGG)$ we get
\begin{align}
 \frac{[k(Y_0):k]}{[k(V):k]}\deg H &\leq \deg V\deg(H+\sbgrpGG)\ll\deg V\deg H\deg \sbgrpGG,\notag\\
\label{bezoutin7.5} [k(Y_0):k] &\ll [k(V):k]\deg V\deg \sbgrpGG,
\end{align}

 and this gives the desired bound.

 Substituting \eqref{degG75} in \eqref{bezoutin7.5}, and using all the bounds from Theorem \ref{trasla} we obtain 
 \begin{equation*}
 \begin{split}
[k(Y_0):\mathbb{Q}]\ll_\eta [k(V):k]^{r+\eta} ({\deg V})^{3r-1} (h(V)+\deg V)^{\frac{(2r-1)r_1+r(r-1)}{\codim V-1}+\eta}\cdot\\
\cdot {[k_\mathrm{tor}(V):k_\mathrm{tor}]}^{\frac{(3r-2)r_1}{\codim V-1}}.
\end{split}
\end{equation*}

 Finally, as in the proof of Theorem \ref{tadimzero}, $\zeta$ can be chosen so that  $[k(\zeta):\mathbb{Q}] \ll  [k(Y_0): \mathbb{Q}]$; again by Serre's theorem,   the points $\zeta$ have order bounded by $\mathrm{ord}(\zeta)\ll_\eta [k(Y_0):\mathbb{Q}]^{\frac{N}{2}+\eta}$ and therefore belong to an  explicit  finite set. 

 Finally, the set of all possible translates $H+p$ has cardinality $S$ bounded by

\begin{equation}
 S\ll_\eta [k(V):k]^{D_1 + \eta} ({\deg V})^{D_2}(h(V)+\deg V)^{D_3+\eta} {[k_\mathrm{tor}(V):k_\mathrm{tor}]}^{D_4}
\end{equation}
where
\begin{align*}
D_1&=\frac{rN(2N+1)}{2}\\
D_2&=\frac{(3r-1)N(2N+1)}{2}+1\\
D_3&=\frac{(N+1)r}{\codim V -1} + \frac{N(2N+1)}{2}\frac{(2r-1)(r-\codim V +1)+r(r-1)}{\codim V -1}\\
D_4&=\frac{(N+1)r}{\codim V -1} +\frac{N(2N+1)}{2}\frac{(3r-2)(r-\codim V +1)}{\codim V-1}.\qedhere
\end{align*}
\end{proof}

\section*{Acknowledgements}
We warmly thank the referee for a  thorough report, precious comments and for the possibility to give some more details. We wish to thank P. Philippon for useful discussions and for {the proof of Lemma \ref{minimoess}.} We are grateful to U. Zannier for his valuable comments and suggestions on an early version of this paper. We  kindly thank A. Galateau for his communication of the constants in \cite{galateau}.  We thank D. Bertrand,  P. Habegger, M. Hindry and D.~Masser for providing us useful references.

We warmly thank the FNS for the great support  we receive in doing  research in Mathematics.
E. Viada also likes to thank the AIM and the organizers J. Pila and U.~Zannier for supporting her participation to a conference on unlikely intersections  in Pisa, where this work was inspired.
S. Checcoli and F. Veneziano would also like to address a special thank to their advisor U. Zannier for introducing them to mathematical research.

\def\cprime{$'$}
\providecommand{\bysame}{\leavevmode\hbox to3em{\hrulefill}\thinspace}
\providecommand{\MR}{\relax\ifhmode\unskip\space\fi MR }
\providecommand{\MRhref}[2]{%
  \href{http://www.ams.org/mathscinet-getitem?mr=#1}{#2}
}
\providecommand{\href}[2]{#2}

 \end{document}